%% file: mNash20240626.tex
\documentclass[11pt]{amsart}

\usepackage{fullpage, amsmath, amssymb, mathrsfs, url, color, soul, stmaryrd, mathtools, bm}
\usepackage[shortlabels]{enumitem}
\usepackage[all]{xy}
\usepackage{tikz}
\usepackage{tikz-cd}

\definecolor{hot}{RGB}{65,105,225}
\usepackage[pagebackref=false, colorlinks=true, linkcolor=hot , citecolor=hot, urlcolor=hot]{hyperref}

\usepackage{pgfplots}              
\pgfplotsset{compat = newest}
\pgfplotsset{ticks=none}

\usepackage{subfigure}             

\usepackage{import}
\usepackage{xifthen}
\usepackage{pdfpages}
\usepackage{transparent}

\def\bB{\mathbb{B}}
\def\bR{\mathbb{R}}
\def\bZ{\mathbb{Z}}
\def\Z{\mathbb{Z}}
\def\bQ{\mathbb{Q}}
\def\Q{\mathbb{Q}}
\def\bN{\mathbb{N}}

\def\R{\mathbb{R}}
\def\C{\mathbb{C}}
\def\F{\mathbb{F}}
\def\bC{\mathbb{C}}

\def\bP{\mathbb{P}}

\def\cC{\mathcal{C}}

\def\cO{\mathcal{O}}
\def\cQ{\mathcal{Q}}

\def\lam{\lambda}
\def\cL{\mathscr L}
\def\LL{\mathscr L}

\def\sX{\mathscr X}
\def\X{\mathscr X}

\def\ord{{\rm ord}}
\def\ol{\overline}
\def\be{\begin{equation}}
\def\ee{\end{equation}}
\def\al{\alpha}
\def\da{\dashrightarrow}
\def\xa{\xrightarrow}
\def\lct{{\rm lct}}
\def\codim{{\rm codim}}
\def\cJ{\mathcal{J}}
\def\eps{\epsilon}
\def\lQ{\mathop{\;\sim\;}_\bQ}
\def\noi{\noindent}
\def\sm{\smallskip}
\def\Ex{{\rm Ex}}

\DeclareMathOperator{\Spec}{Spec}
\DeclareMathOperator{\mult}{mult}
\DeclareMathOperator{\Hom}{Hom}
\DeclareMathOperator{\Div}{div}
\DeclareMathOperator{\Supp}{Supp}

\newtheorem{theorem}{Theorem}[section]
\newtheorem{thm}[theorem]{Theorem}

\newtheorem{cor}[theorem]{Corollary}
\newtheorem{lemma}[theorem]{Lemma}
\newtheorem{proposition}[theorem]{Proposition}
\newtheorem{prop}[theorem]{Proposition}

\newtheorem{corollary-definition}[theorem]{Corollary and Definition}

\theoremstyle{definition}

\newtheorem{defn}[theorem]{Definition}
\newtheorem{nota}[theorem]{Notation}

\newtheorem{rmk}[theorem]{Remark}

\newtheorem{subs}[theorem]{ }

\long\def\comment#1{}

\begin{document}

\title{On the embedded Nash problem}

\author{Nero Budur}\address{Department of Mathematics, KU Leuven, Celestijnenlaan 200B, 3001 Leuven, Belgium, and, YMSC, Tsinghua University, 100084 Beijing, China, and BCAM, Mazarredo 14, 48009 Bilbao, Spain.}
\email{nero.budur@kuleuven.be}

\author{Javier de la Bodega}\address{Department of Mathematics, KU Leuven, Celestijnenlaan 200B, 3001 Leuven, Belgium, and BCAM, Mazarredo 14, 48009 Bilbao, Spain}
\email{javier.delabodega@kuleuven.be}

\author{Eduardo de Lorenzo Poza}\address{Department of Mathematics, KU Leuven, Celestijnenlaan 200B, 3001 Leuven, Belgium, and BCAM, Mazarredo 14, 48009 Bilbao, Spain}
\email{eduardo.delorenzopoza@kuleuven.be}

\author{Javier Fern\'andez de Bobadilla}\address{Ikerbasque, Basque Foundation for Science, Maria Diaz de Haro 3, 48013, Bilbao, Spain, and BCAM, Mazarredo 14, 48009 Bilbao, Spain, and Academic Collaborator at UPV/EHU.} \email{jbobadilla@bcamath.org}

\author{Tomasz Pe\l ka}\address{BCAM, Mazarredo 14, 48009 Bilbao, Spain}
\email{tpelka@bcamath.org}




\begin{abstract}
The embedded Nash problem for a hypersurface in a smooth algebraic variety is to characterize geometrically the maximal irreducible families of arcs with fixed order of contact along the hypersurface.
We show that divisors on minimal models of the pair contribute with such families. We solve the problem for unibranch plane curve germs, in terms of the resolution graph. These are embedded analogs of known results for the classical Nash problem on singular varieties. 
\end{abstract}

\maketitle

\section{Introduction}

We address the embedded Nash problem for  a hypersurface in a smooth algebraic variety. There is some analogy with the classical Nash problem which we recall now.

\subs {\bf The classical Nash problem.}  Let $V$ be a complex algebraic variety. To every maximal irreducible family of arcs through the singular locus of $V$ one can associate a divisorial valuation, called a {\it Nash valuation}. The classical Nash problem is to characterize  the Nash valuations in terms of resolutions of  singularities of $V$. An {\it essential valuation} is a divisorial valuation
whose center on every resolution of $V$ is an irreducible component of the
inverse image of the singular locus of $V$. A divisorial valuation over $V$ is called a {\it terminal valuation} if it is given by some  exceptional divisor on some minimal model over $V$ of a resolution of $V$. The name comes from the fact that minimal models have terminal singularities, a mild type of singularities. One has:

\begin{thm}\label{thmCl} Let $V$ be a complex algebraic variety. 

(i) (de Fernex-Docampo \cite{dFD}) There are  inclusions
$$
\{\text{terminal valuations}\}\subset\{\text{Nash valuations}\}\subset\{\text{essential valuations}\}.
$$

(ii) (Fern\'andez de Bobadilla-Pe Pereira \cite{FdBPP}) If $\dim V =2$ the three sets are equal.
\end{thm}

\noi
See  \cite{BLM} for further results. If $\dim V=2$, essential valuations are terminal valuations, due to existence of minimal resolutions. Thus (i) implies (ii), which  solves the Nash problem in this case. 

\subs\label{submN} {\bf The embedded Nash problem.} Let $(X,D)$ be a pair consisting of a smooth complex algebraic variety $X$ and a non-zero effective divisor $D$. 
The interesting arcs on $X$  are now the ones with prescribed order of contact with $D$. Contact loci of arcs are fundamental for motivic integration and appear in the monodromy conjecture, see Denef-Loeser \cite{DL}. They also appear in a conjectural algebraic description of the Floer cohomology of iterates of the monodromy by Budur-Fern\'andez de Bobadilla-L\^e-Nguyen \cite{Fl}. 

Every maximal irreducible family of arcs in $X$ with fixed positive order of contact with $D$ is obtained from some prime divisor on a log resolution of $(X,D)$, by Ein-Lazarsfeld-Musta\c{t}\u{a} \cite{ELM}, and thus one can associate to it a divisorial valuation, which we call a {\it contact valuation}. The embedded Nash problem, posed in \cite[\S 2]{ELM}, is the geometric characterization of the contact valuations in terms of log resolutions  of $(X,D)$.

There is a priori no reason why the Nash problem for a pair should be related with the classical Nash problem. Indeed, the arc space of $D$ is complementary in the arc space of $X$ to the set of arcs with finite order of contact with $D$. Nevertheless, we show that Theorem \ref{thmCl} has an analog for pairs. We give now a gentle introduction to the results. Precise statements are given in \ref{subsMrd}.

We introduce {\it essential valuations} for a pair $(X,D)$ in terms of log resolutions satisfying a separating property which has appeared in work of Nicaise-Sebag \cite{NS1, NS2}, McLean \cite{Mc}, and in \cite{Fl}.  If $D$ is reduced, we introduce {\it divisorial log terminal valuations} (or {\it dlt valuations}) as those given by prime divisors lying over $D$ in dlt modifications of $(X,D)$.  Dlt modifications are minimal models over $X$ obtained from log resolutions of $(X,D)$ and have dlt singularities, a mild type of singularities. Their existence is proved by Odaka-Xu \cite{OX}.   

Our first main result is that, for $X$ smooth and $D$ a reduced effective divisor, there are inclusions
$$
\{\text{dlt valuations}\}\subset\{\text{contact valuations}\}\subset\{\text{essential valuations}\}.
$$
 The  proof is based on a variation of an argument from \cite{Fl}, where  positivity of a relative ample divisor on a log resolution is exchanged with that of the log canonical divisor on a dlt modification. Positivity is used to estimate intersection numbers with the arcs in a 1-parameter family of arcs, a theme  common to  \cite{FdBPP, dFD, Fl}.

Among the dlt valuations are the  
 {\it jump valuations}, that is, valuations defined by prime divisors computing jumping numbers of $(X,D)$. The {\it top contact valuations}, that is, those contact valuations which correspond to minimal-codimension irreducible components of loci of arcs with fixed contact order, can  be determined numerically from log resolutions.

For an irreducible formal plane curve germ, the three sets of valuations are not equal, unlike in the dimension-two case of the classical Nash problem.  Our second main result is the characterization in terms of the resolution graph of the set of contact valuations. This solves the embedded Nash problem in this case. The main idea, besides  minimal separating log resolutions,  is to use the topology of the Milnor fibration in terms of Dehn twists to rule out adjacencies between strata of contact loci. It turns out that the irreducible components of contact loci in this case are  disjoint.

\subs{\bf Previous work.} The  embedded Nash problem is solved for an affine toric variety given by a cone together with a toric invariant ideal by Ishii \cite{Is-t},   and for  hyperplane arrangements by Budur-Tue \cite{BT}, see Section \ref{secHA}. Closely related to the embedded Nash problem is the {\em generalized Nash problem}, introduced in~\cite{Is}, which consists in studying adjacencies between maximal divisorial sets; the case of maximal divisorial sets in the plane was treated in~\cite{FPPP}.

A related problem is the characterization of irreducible components of jet schemes of a singular variety. Known results in this direction provide numerical characterizations in  cases with combinatorial flavor. This gives sometimes that all contact valuations are top, see Section \ref{secOth}, e.g.  hypersurfaces with  maximal-multiplicity rational singularities by Mourtada \cite{Mou-rat}, Bruschek-Mourtada-Schepers \cite{BMS},  and determinants of generic square matrices by Docampo \cite{Roi}. In general, the passage from a numerical characterization of the configuration of irreducible components of the jet schemes of $D$ to a geometric characterization (that is, in terms of valuations) of the irreducible components of the contact loci of $(X,D)$ is more difficult, and only partial solutions to the embedded Nash problem are obtained, see \cite{Mou-rat, Mou-tor, Mou-Pl,  KMPT, CoM, Kor}.

A numerical algorithm to determine the configuration of irreducible components of the jet schemes of a unibranch curve singularity was given by Mourtada \cite{Mou}. This gave information in terms of the resolution graph only for certain divisors in Lejeune-Jalabert-Mourtada-Reguera \cite[Thm. 2.7]{LJMR}. Our solution of the embedded Nash problem in this case does not use \cite{Mou}.


\subs {\bf The results in detail.} \label{subsMrd}
Let $\cL (X)=\Hom_{\bC\rm{-sch}}(\Spec\bC\llbracket t \rrbracket, X)$ be the arc space of $X$. Let $\Sigma\subset D_{red}$ be a non-empty Zariski closed subset. Let $m\ge 1$ be an integer.

\begin{defn}\label{defsX}   The {\it $m$-contact locus of $(X,D,\Sigma)$} is
$$
\sX_m (X,D,\Sigma) :=\{\gamma\in \cL (X)\mid \gamma (0) \in \Sigma\text{ and } \ord_\gamma D = m\}.
$$
For simplicity we use the notation $\sX_m $ when the context is clear. We set $\sX_m (X,D):=\sX_m (X,D,D_{red})$.
\end{defn}

\begin{defn}\label{eqSm}
 Let $\mu:Y\to X$ be a 
{\it log resolution} of $(X,D,\Sigma)$, that is, $\mu$ is a projective birational morphism from a smooth variety $Y$ such that $\mu^{-1}(D)$ and $\mu^{-1}(\Sigma)$ are simple normal crossings divisors. Write the decomposition into irreducible components 
$$\Supp(\mu^{-1}(D))=\cup_{i\in S}E_i$$
and let
$$
N_i:=\ord_{E_i} D,\quad \nu_i:=1+\ord_{E_i}(K_{Y/X}).
$$
We set
$$
S_m:=\{i\in S\mid  \mu(E_i)\subset \Sigma \text{ and }N_i\text{ divides }m \}.
$$
For every $i\in S$, we set
$$
\sX_{m,i}  :=\{\gamma\in \sX _m\mid \gamma \text{ lifts to an arc on }Y\text{ centered on }E_i^\circ\}
$$
where $E_i^\circ$ is the complement of $\cup_{j\in S\setminus\{i\}}E_j$ in $E_i$. Note that the lift of the arc is unique since $m$ is finite, by the valuative criterion of properness.
\end{defn}

\begin{defn}
A log resolution $\mu:Y\to X$ of $(X,D,\Sigma)$ is called {\it $m$-separating}  if in addition $\mu$ is an isomorphism over $X\setminus D$, and for any $i\neq j\in S$  with $E_i\cap E_j\neq\emptyset$, one has $N_i+N_j>m$. 
\end{defn}
This notion  appeared with a different name in \cite[Lemma 5.17]{NS1}, \cite{NS2}. The name was coined in \cite{Mc} and used in \cite{Fl}. Starting with a log resolution, one can blowup further smooth centers until an $m$-separating one is obtained, see \cite[Lemma 2.9]{Fl}.
If $\mu$ is an $m$-separating log resolution of $(X,D,\Sigma)$, by \cite{ELM} there is a partition into non-empty smooth irreducible  locally closed subvarieties, see also Proposition \ref{propCli}:
\be\label{eqD}
\sX _m=\bigsqcup_{i\in S_m}\sX_{m,i} 
\ee

\begin{defn} 
An {\it $m$-valuation of $(X,D,\Sigma)$} is a divisorial valuation $v$ on $X$ given by $E_i$, that is, $v=\ord_{E_i}$, for some $i\in S_m$ with respect to some  log resolution (equivalently, some $m$-separating log resolution) $\mu$ of $(X,D,\Sigma)$. 
\end{defn}

Thus every irreducible component of $\sX _m$ arises as the Zariski closure of some
$\sX_{m,i} $ with respect to some $m$-separating log resolution, and hence it corresponds to a unique $m$-valuation of $(X,D,\Sigma)$. 

\begin{defn} A {\it contact $m$-valuation of $(X,D,\Sigma)$} is an $m$-valuation corresponding to an irreducible component  $\ol{\sX_{m,i} }$ of the $m$-contact locus $\sX _m$ of $(X,D,\Sigma)$. A {\it contact valuation of $(X,D,\Sigma)$} is  a contact $m$-valuation for some $m\ge 1$. A contact valuation can be an $m'$-valuation without being a contact $m'$-valuation.
\end{defn}

\begin{defn} An
{\it essential $m$-valuation of $(X,D,\Sigma)$} is an $m$-valuation with center a prime divisor lying over $\Sigma$ on every $m$-separating log resolution of $(X,D,\Sigma)$. An {\it essential valuation of $(X,D,\Sigma)$} is an essential $m$-valuation for some $m\ge 1$. An essential valuation can be an $m'$-valuation without being an essential $m'$-valuation. Contact $m$-valuations are essential $m$-valuations  by (\ref{eqD}). 
\end{defn}


\begin{defn} 
Assume $D$ is reduced. A {\it divisorial log terminal valuation} (or {\it dlt valuation}) of $(X,D,\Sigma)$ is a valuation of $X$ given by a prime divisor lying over $\Sigma$ in a dlt modification of $(X,D)$ (see Definition \ref{defMM}). A {\it dlt $m$-valuation of $(X,D,\Sigma)$} is a dlt valuation that is also an $m$-valuation.
\end{defn}



\begin{thm}\label{thmM} Let $X$ be a smooth complex algebraic variety, $D$ a non-zero reduced effective divisor on $X$, and $\Sigma$ a Zariski closed subset of $D$. Let $m\ge 1$ be an integer. Then for $(X,D,\Sigma)$,
\be\label{eqInc}
\{\text{dlt }m{-valuations }\}\subset\{\text{contact }m{-valuations }\}\subset\{\text{essential }m{-valuations}\}.
\ee
\end{thm}


For hyperplane arrangements, all three sets in (\ref{eqInc}) are equal if the log resolution satisfies an additional property, see Section \ref{secHA}.

\begin{thm}\label{corArr}
If $D$ is a hyperplane arrangement in $X=\bC^n$, and $m\in\bZ_{>0}$, then for $(X,D)$ the three sets of $m$-valuations in (\ref{eqInc}) are equal, and consist of the divisorial valuations corresponding to $E_i$ with $i\in S_m$ as in Definition \ref{eqSm} for any good $m$-separating log resolution of $(X,D)$.
\end{thm}

 Next we show that some interesting valuations lie in the sets from (\ref{eqInc}).  
 
 \begin{defn}
 A {\it jump valuation of $(X,D,\Sigma)$} is a valuation given by a prime divisor lying over $\Sigma$ in a log resolution of $(X,D,\Sigma)$  that contributes with a jumping number to $(X,D)$, see Definition \ref{defMJ}.  
 \end{defn}
 
 Adapting the proof of a result of Smith-Tucker \cite[Appendix]{BV} we show:
  
\begin{thm}\label{thmJV} Let $(X,D,\Sigma)$ be as in Theorem \ref{thmM}. If $\mu$ is a log resolution of $(X,D,\Sigma)$ and $E_i$ with $i\in S$ defines a jump valuation of $(X,D,\Sigma)$, then $E_i$ defines a dlt valuation, and hence a contact $m$-valuation, of $(X,D,\Sigma)$ for every $m$ divisible by $N_i$.
\end{thm}

The converse does not always hold, see Baumers-Veys \cite[\S 7]{BV}. 

\begin{defn}
The {\it top contact $m$-valuations} of $(X,D,\Sigma)$ are the contact $m$-valuations corresponding to the minimal-codimension irreducible components of $\sX _m$. 
\end{defn}

\begin{defn}\label{defmlct}
If $D$ is  an effective divisor, and $\mu$ is an $m$-separating log resolution of $(X,D, \Sigma)$,  the {\it $m$-log canonical threshold} of $(X,D,\Sigma)$ is
$$
\lct_m(X,D,\Sigma) :=\min\left\{{\nu_i}/{N_i}\mid i\in S_m\right\}.
$$
\end{defn}
The right-hand side depends on the log resolution if the $m$-separating condition is dropped.  Nevertheless, $\lct_m(X,D,\Sigma)$ is independent of $\mu$ by the following, essentially  contained in \cite{M}, \cite{ELM}, where it was stated for the usual log canonical thresholds:

\begin{prop}\label{corLct} The codimension of $\sX _m$ is
$m \, \lct_m(X,D,\Sigma)$. The set of top contact $m$-valuations is the set of $m$-valuations (necessarily essential $m$-valuations) $E_i$ such that  $\lct_m(X,D,\Sigma)=\nu_i/N_i$.
\end{prop}

\begin{rmk} (i) Despite the similar appearance, $\lct_m(X,D,\Sigma)$ is not necessarily equal to the {\it usual log canonical threshold  of $(X,D)$ in a neighborhood of $\Sigma$}, 
$$\lct_\Sigma(X,D):=\min\left\{{\nu_i}/{N_i}\mid i\in S\text{ and } \mu(E_i)\cap \Sigma\neq\emptyset \right\}.
$$
In general, $\lct_\Sigma(X,D)\le \lct_m(X,D,\Sigma)$ and the inequality can be strict.

 (ii) Every prime divisor $E_i$ lying over $\Sigma$ in some log resolution of $(X,D,\Sigma)$ such that $E_i$ computes $\lct_\Sigma(X,D)$, gives a top contact $m$-valuation of $(X,D,\Sigma)$ for all $m$ divisible by $N_i$. 
 
 (iii) The  numbers $\lct_m(X,D,\Sigma)$ and the number of top contact $m$-valuations of $(X,D,\Sigma)$ for all $m\ge 1$, are read from a single log resolution $\mu$ of $(X,D,\Sigma)$ by applying the virtual Poincar\'e polynomial realization to the expression in terms of $\mu$ of the motivic zeta function of $(X,D,\Sigma)$, cf. \cite[Ch. 2, Corollary 3.5.12.]{ACL}.
\end{rmk}

Next we look at  cases when $D$ is has mild singularities.

\begin{thm}\label{thmLcv}
Suppose $D$ is reduced and $(X,D)$ is log canonical. For $(X,D,\Sigma)$:

(i) If $\lct_m(X,D,\Sigma)=1$, then
$
\{\text{dlt } m\text{-valuations}\} =\{\text{top contact } m\text{-valuations}\}.
$

(ii) If $\lct_m(X,D,\Sigma)\neq 1$, then $
\{\text{dlt } m\text{-valuations}\} =\emptyset.
$

(iii) If $\Sigma=D$, then for every $m\ge 1$
$$
\{\text{dlt } m\text{-valuations}\} = \{\text{contact } m\text{-valuations}\} = \{\text{top contact } m\text{-valuations}\}.
$$

(iv) The last three sets are singletons, consisting of the valuation defined by $D$, if in addition $D$ has rational singularities. 

\end{thm}

Part (ii) implies that, if a singular hypersurface $D$ in a smooth variety $X$ has rational singularities, then there are no dlt valuations for $(X,D,D_{sing})$. One wonders then if there is chaos or order among the contact valuations in this case. We address some examples in Section \ref{secOth}.

Assume now that $X$ is a smooth surface. For simplicity, we state the results for an irreducible formal plane curve germ singularity $(\bC^2,C,0)$.  A minimal $m$-separating log resolution exists. Its resolution graph $\Gamma_m$ is a refinement of $\Gamma_{m-1}$, easily obtained by inserting vertices on edges until the $m$-separating condition is satisfied;  $\Gamma_1$ is the resolution graph of the minimal log resolution. The essential $m$-valuations correspond to the subset $S_m$ (as in Definition \ref{eqSm}) of vertices of $\Gamma_m$. 

We label the rupture vertices by $E_{R_1},\ldots, E_{R_g}$, in the order of their appearance in the resolution process. Here $g\ge 1$ is the number of Puiseux pairs of the germ. The strict transform of $C$ is denoted by $\tilde C$. We denote by $Z_j$, for $j=1,\ldots, g$, the subset of arcs in the contact locus $\sX_m=\sX_m(\bC^2,C,0)$ which lift to one of the divisors from the vertical groups in $\Gamma_m$ as in Figure \ref{figr}. Denote by $S_m'$ the remaining vertices in $\Gamma_m$ which also lie in $S_m$. We let $\sX_{m,i}$ for $i\in S_m$ be as in (\ref{eqD}).

\begin{figure}[ht] 
    \centering
    \resizebox{.7\linewidth}{!}{%
    \fontsize{20pt}{20pt}\selectfont
    \def\svgheight{0.5cm}
    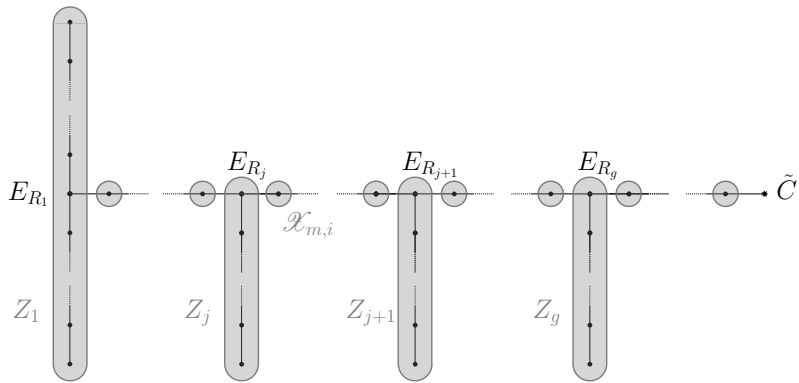
}
    \caption{Resolution graph $\Gamma_m$ for the minimal $m$-separating log resolution, with the irreducible components (if non-empty) of the $m$-contact locus $\sX_m$  in gray.}
    \label{figr}
\end{figure}

\begin{theorem}\label{thmNPCur} Let $(\bC^2,C,O)$ be an irreducible formal plane curve germ singularity with $g\ge 1$ Puiseux pairs. Let $m\ge 1$ be an integer. Then 
$$
\sX_m = \bigsqcup_{j=1}^g Z_j \sqcup  \bigsqcup_{i\in S_m'} \sX_{m,i}
$$
is a disjoint union decomposition such that, in the Zariski topology:

(i) Each non-empty set in the decomposition is  irreducible and a connected component of $\sX_m$.

(ii) $\sX_{m,i}$ is non-empty for $i\in S_m'$.

(iii) $Z_j$ is non-empty if and only if there exists a vertex $i\in S_m$ in the $j$-th vertical group in Figure \ref{figr}.  If $E_{R_1}$ is not in $S_m$, then the vertices with $i \in S_m$ in the first vertical group can only appear on one side of $E_{R_1}$. If $Z_j$ is non-empty, the unique vertex $i_0\in S_m$ in the $j$-th vertical group that is the closest to the rupture divisor $E_{R_j}$ (it could be $E_{R_j}$ itself), satisfies $Z_j = \overline{\sX_{m,i_0}}$.
\end{theorem}

This solves the embedded Nash problem for irreducible formal plane curve germs. 
The dlt $m$-valuations correspond to the points in $S_m$  on the horizontal path from $E_{R_1}$ to $\tilde C$ in  $\Gamma_m$ by Proposition \ref{propDltCurves}.

\sm

\noi
{\bf Organization.}
In Section \ref{secMM} we recall facts about relative minimal models. In Section \ref{secMsep} we recall the arc spaces and contact loci, and prove Theorem \ref{thmM}, Theorem \ref{thmJV}, Proposition \ref{corLct}, Theorem \ref{thmLcv}. In Section \ref{secHA} we address  hyperplane arrangements. In Section \ref{secOth} we look at some  hypersurfaces with rational singularities. In Section \ref{secCur} we characterize the essential and the dlt valuations in terms of the resolution graph for curves on smooth surfaces. In Section \ref{secCNP} we prove Theorem \ref{thmNPCur}.

\sm

\noi{\bf Acknowledgement.} We thank:  H. Mourtada, M. Musta\c{t}\u{a}, J. Nicaise, K. Palka, N. Potemans, M. Saito, R. van der Veer, W. Veys,  and C. Xu for comments, and IHES and UIB Palma for hospitality. We are especially grateful to a referee who pointed out a gap in the proof of Theorem \ref{decomp} in the first  version of this article. The authors were supported by the grants: Methusalem METH/15/026, G097819N, G0B3123N from FWO,  SEV-2017-0718-19-1, PRE2019-087976, PID2020-114750GB-C33 from Spanish Ministry of Science,  BERC 2018-2021 and IT1094-16 from  Basque Government. This work started at a research in pairs at CIRM in the framework of a Jean Morlet Semester, where the authors enjoyed excellent working conditions.

\section{Relative minimal models}\label{secMM}

We recall some basic terminology and facts about  minimal models, cf. \cite{KM, K}. A {\it divisor} on a normal complex variety $X$ will mean a Weil divisor. If we allow $\bQ$-coefficients, we refer to it as a {\it $\bQ$-divisor}. If $D_1, D_2$ are $\bQ$-divisors, $D_1\ge D_2$ means that the inequality is satisfied by the coefficients of each prime divisor. If $D$ is a divisor, recall that $\cO_X(D)$ denotes the sheaf of rational functions $g$ with $\Div (g)+D\ge 0$. It is a reflexive sheaf of rank 1, and it is locally free if and only if $D$ is Cartier. We say that $X$ is {\it $\bQ$-factorial} if every $\bQ$-divisor is $\bQ$-Cartier.

\begin{defn}
Let $(X,D)$ be a pair consisting of a normal complex variety $X$ and a $\bQ$-divisor $D$ on $X$ such that $K_X+D$ is $\bQ$-Cartier. If $E$ is a prime divisor over $X$, the {\it log discrepancy of $E$ with respect to $(X,D)$}, denoted $a(E,X,D)$, is defined as follows. Let  $\mu:Y\to X$ be a birational morphism from a normal variety $Y$. Then there is a $\bQ$-linear equivalence of $\bQ$-Cartier $\bQ$-divisors 
$$
K_Y+ \mu_*^{-1}D + \Ex(\mu) \lQ \mu^*(K_X+D) +\sum_{i\in I}a(E_i,X,D) E_i
$$
for some $a(E_i,X,D)\in \bQ$, where $\mu_*^{-1}$ denotes the strict transform, and $\Ex(\mu)=\sum_{i\in I}E_i$ is the sum of the exceptional prime divisors with coefficients 1. If $E=E_i$ for some $\mu$, define $a(E,X,D)=a(E_i,X,D)$.  If $E$ is an irreducible component of $D$, define $a(E,X,D)=1-\ord_{E}D$. For all other prime divisors $E$ over $X$ set $a(E,X,D)=0$.    This definition is independent of the choice of $\mu$. Note that $a(E_i,X,D)=\nu_i-N_i$ for all $i\in I$, where $\nu_i-1=\ord_{E_i}(K_{Y/X})$ and $N_i=\ord_{E_i}D$. 
\end{defn}

\begin{defn}\label{defDltY}
Let $(X,D)$ be a pair consisting of  a normal variety $X$ and  a $\bQ$-divisor $D=\sum_ia_iD_i$ such that $K_X+D$ is $\bQ$-Cartier. 
We assume that $D$ is a {\it boundary}, that is, $0\le a_i\le 1$ for all $i$. Then we say that $(X,D)$ is a {\it divisorial log terminal pair} (or {\it dlt pair}) if there is a closed subset $Z\subset X$ such that $X\setminus Z$ is smooth, $D_{| X\setminus Z}$ is a simple normal crossings divisor, and if for a (equivalently, for all)  log resolution  $\mu:Y\to X$ of $(X,D)$ with $\mu^{-1}(Z)$  pure of codimension 1, the condition $a(E,X,D)>0$ is satisfied for every prime divisor $E\subset \mu^{-1}(Z)$. In other words, a dlt pair is snc outside the {\it klt locus} $Z$.  We say that $(X,D)$ is a {\it log canonical pair} if $a(E,X,D)\ge 0$ for all prime divisors $E$ over $X$. In particular, a dlt pair is  log canonical.
\end{defn}

\begin{defn} Let $\mu:Y\to X$ be a projective morphism between normal varieties, and let $K_Y+\Delta$ be a $\bQ$-Cartier $\bQ$-divisor. Then $K_Y+\Delta$ is {\it semiample over $X$} if there exist a  morphism $\phi:Y\to Z$ to a normal variety $Z$ over $X$, and a $\bQ$-divisor $A$ on $Z$ which is ample over $X$, such that $K_Y+\Delta\sim_\bQ\phi^*A$. \end{defn}

\begin{defn}\label{defMM}

Let $(Y,\Delta)$ be a dlt pair, $\mu:Y\to X$ a projective morphism of complex algebraic varieties, and $D$ a boundary on $X$.

(i) We say that $(Y,\Delta)$ is a {\it  minimal model} (respectively {\it good minimal model})  {\it over $X$} if $K_Y+\Delta$ is $\mu$-nef (resp. $\mu$-semiample). 

(ii) We say that $(Y,\Delta)$ is a {\it dlt modification} (resp. {\it good dlt modification}) {\it of $(X,D)$} if it is a (good) minimal model over $X$ and $\Delta=\mu_*^{-1}D+ \Ex(\mu)$.

(iii) Let $(Y',\Delta')$ be a pair sitting in a diagram
\be\label{eqMM2}
\xymatrix{
(Y,\Delta)  \ar[dr]_\mu \ar@{-->}[rr]^{\phi}& & (Y',\Delta') \ar[dl]^{\mu'} \\
& X. &
}
\ee
We say $(Y',\Delta')$ is a {\it minimal model} (resp. {\it good minimal model})  {\it of $(Y,\Delta)$ over $X$} if:

\begin{enumerate}[(a)]
\item $\phi:Y\dashrightarrow Y'$ is a birational contraction, that is, $\phi^{-1}$ has no exceptional divisors,
\item $\mu':Y'\to X$ a projective morphism, and $\mu'\circ\phi=\mu$ as birational maps,
\item $Y'$ is a  normal variety,
\item $\Delta'=\phi_*\Delta$,
\item $a(E,Y,\Delta)<a(E,Y',\Delta')$ for every $\phi$-exceptional divisor $E\subset Y$, 
\item $K_{Y'}+\Delta'$ is $\bQ$-Cartier and $\mu'$-nef (resp. $\mu'$-semiample).
\end{enumerate}

\end{defn}

The conditions (a)-(f) imply that $a(E,Y,\Delta)\le a(E,Y',\Delta')$ for all divisors $E$ over $Y$ by \cite[Proposition 1.22]{K}, but it is not automatically clear if $(Y',\Delta')$ is dlt.   

\begin{thm}\label{thmLocIso} (\cite[Corollary 1.23]{K}) If $(Y',\Delta')$ is as in Definition \ref{defMM} and it is obtained by running the MMP for $(Y,\Delta)$ over $X$, then:

(i) $(Y',\Delta')$ is a dlt pair, and

(ii) For any prime divisor $E$ over $Y$, $a(E,Y,\Delta)<a(E,Y',\Delta')$ iff $\phi$ is not a local isomorphism at the generic point of the center of $E$ on $Y$. 

\end{thm}

Regarding the validity of the minimal model program one has the following. While it is not stated in the same way as in \cite{OX}, the proof of the main result contains it.

\begin{thm}\label{thmEMM} (Odaka-Xu \cite{OX})  Let $X$ be a normal quasi-projective variety  and $D$ a boundary on $X$ such that $K_X+D$ is $\bQ$-Cartier (resp. $X$ is also $\bQ$-factorial). For any log resolution $\mu:Y\to X$ of $(X,D)$, the minimal model program for $(Y,\Delta=\mu_*^{-1}D+\Ex(\mu))$ over $X$ runs, that is, flips exist, and terminates with  good dlt minimal models (resp.  good dlt $\bQ$-factorial minimal models) of $(Y,\Delta)$ over $X$.
\end{thm}

One can characterize which  divisors get contracted on minimal models using stable base loci:

\begin{defn}
Let $\mu:Y\to X$ be a projective morphism of normal varieties. Assume that  $X$ is affine (for our applications this case suffices).
If $\Delta$ is a $\bQ$-Cartier $\bQ$-divisor on $Y$, the {\it stable base locus of $\Delta$ over $X$} is $$\bB(\Delta/X)=\bigcap_{k\in \bN,\, k\Delta \text{ is Cartier}}\quad\bigcap_{s\in H^0(Y,\cO_Y(k\Delta))}\Supp(\Div(s)).$$ 
\end{defn}

\begin{prop}\label{propBE}(\cite[1.21]{K}, \cite[Lemma 2.4]{HX})
Let $\mu:Y\to X$ be a projective morphism of varieties, $(Y,\Delta)$ a dlt pair, and $\phi:Y\da Y'$, $\psi:Y\da Y''$ two minimal models of $(Y,\Delta)$ over $X$. Then:

(a)  $Y'\da Y''$ is an isomorphism in codimension 1 such that for all prime divisors $E$ over $X$ one has $a(E,Y',\phi_*\Delta)=a(E,Y'',\psi_*\Delta)$.

(b) If $(Y',\phi_*\Delta)$ is a dlt good minimal model of $(Y,\Delta)$ over $X$, so is $(Y'',\psi_*\Delta)$.

(c) If $(Y',\phi_*\Delta)$ is a dlt good minimal model of $(Y,\Delta)$ over $X$, then the set of $\phi$-exceptional divisors is the set of divisors contained in the stable base locus $\bB(K_Y+\Delta/X)$. 
\end{prop}

One can compare minimal models obtained from a tower of projective birational morphisms:

\begin{prop}\label{propFc}(\cite[1.27]{K})  Let $\mu:Y\to X$ be a projective morphism of varieties. Let $\pi:Y_1\to Y$ be a projective birational morphism. Let $\Delta_1$ and $\Delta$ be $\bQ$-divisors on $Y_1$ and $Y$, respectively, such that $\Delta=\pi_*\Delta_1$. Assume that $(Y_1,\Delta_1)$, $(Y,\Delta)$ are dlt pairs and that $a(E,Y_1,\Delta_1)<a(E,Y,\Delta)$ for every prime $\pi$-exceptional divisor $E\subset Y_1$. Then every minimal model of $(Y,\Delta)$ over $X$ is also a minimal model of $(Y_1,\Delta_1)$ over $X$.\end{prop}

A consequence of this proposition is:

\begin{lemma}\label{lemDltm}
Let $X$ be a complex algebraic variety, $D$ a boundary divisor on $X$. Suppose that $(Y',\Delta')$ is a dlt modification of $(X,D)$. Then $(Y',\Delta')$ is a minimal model of $(Y,\Delta)$ over $X$ for any log resolution $\phi:Y\to Y'$ of $(Y',\Delta')$ that is an isomorphism over the non-klt locus of $(Y',\Delta')$, where $\Delta=\phi_*^{-1}\Delta'+\Ex(\phi)$.
\end{lemma}
\begin{proof}
Note that $\Delta$ is also equal to $\mu_*^{-1}D+\Ex(\mu)$, where $\mu=\mu'\circ\phi$. Since $(Y',\Delta')$ is a dlt pair, $a(E,Y',\Delta')>0$ for every $\phi$-exceptional divisor $E$. For such $E$, $a(E,Y,\Delta)=0$. The claim then follows from Proposition \ref{propFc}. 
\end{proof}

The following lemma will be useful for the characterization  of dlt $m$-valuations: 

\begin{lemma}\label{lemNNew}
Let $X$ be smooth affine variety, $D$ a boundary on $X$, $\mu:Y\to X$ a log resolution of $(X,D)$, and $\Delta=\mu_*^{-1}D+\Ex(\mu)$. Let $\pi:Y_1\to Y$ be a projective birational morphism such that $\mu_1=\mu\circ \pi$ is a log resolution of $(X,D)$, and let $\Delta_1=(\mu_1)_*^{-1}D+\Ex(\mu_1)$. Let $F$ be a prime divisor on $Y_1$ such that $E=\pi_*F$ is a non-zero  divisor. Then $F$ does not get contracted on minimal models of $(Y_1,\Delta_1)$ over $X$ if and only if $E$  does not get contracted on  minimal models of $(Y,\Delta)$ over $X$.
\end{lemma}
\begin{proof}
Theorem \ref{thmEMM} applies and gives a good minimal models for $\mu$ and $\mu_1$. Then by Proposition \ref{propBE} all minimal models are good, and we can use stable base loci to describe  which divisors get contracted or not on minimal models. Suppose that $F$ does not get contracted. Then there is some $k\in \bN$ such that $F$ is not in the base locus of $k(K_{Y_1}+\Delta_1)$ over $X$.

Let $G=\sum_{E'}a(E',X,D)E'$ where the sum is over the $\mu$-exceptional prime divisors $E'$ in $Y$, so
$
K_Y+\Delta = \mu^*(K_X+D) +G.
$ 
Let $P=\sum_{F'}a(F',Y,\Delta)F'$ where the sum is over the $\pi$-exceptional prime divisors $F'$ in $Y_1$. Then
$
K_{Y_1}+\Delta_1 = \pi^*(K_Y+\Delta)+ P.
$ 
Then $P$ is effective, since $(Y,\Delta)$ is dlt. Then
$$
K_{Y_1}+\Delta_1 = \mu_1^*(K_X+D) +\pi^*G + P \sim_{\mu_1} \pi^*G+P.
$$
Thus there is a strict inclusion of vector subspaces
$$
H^0(Y_1,\cO_{Y_1}(k(\pi^*G+P)-F))\subsetneq H^0(Y_1,\cO_{Y_1}(k(\pi^*G+P)))
$$
because $F$ is not in the base locus of $k(K_{Y_1}+\Delta_1)$ over $X$. 
By the definition of $\pi_*$, this is the same as
$$
H^0(Y,\pi_*\cO_{Y_1}(k(\pi^*G+P)-F))\subsetneq H^0(Y,\pi_*\cO_{Y_1}(k(\pi^*G+P))).
$$
Now, $$\pi_*\cO_{Y_1}(k(\pi^*G+P))=\cO_Y(kG)\otimes_{\cO_Y}\pi_*\cO_{Y_1}(kP) = \cO_Y(kG),$$
the first equality due to the projection formula, the second equality due to $kP$ being effective and $\pi$-exceptional. Similarly,
$$
\pi_*\cO_{Y_1}(k(\pi^*G+P)-F) =\cO_Y(kG)\otimes_{\cO_Y}\pi_*\cO_{Y_1}(kP-F) = \cO_Y(kG-E),$$
tha last equality due to $F$ being the strict transform of $E$. 
Thus the above strict inclusion is 
$$
H^0(Y,\cO_Y(kG-E))\subsetneq H^0(Y,\cO_Y(kG)),
$$
and so $E$ is not contained in $\bB(G/X)=\bB(K_Y+\Delta/X)$. Hence $E$ does not get contracted on minimal models of $(Y,\Delta)$ over $X$. Also, it is clear that the proof runs in the converse direction as well. 
\end{proof}

\section{Dlt and contact valuations}\label{secMsep}

\subs{\bf Arcs, jets, and contact loci.}  For $l\in\bN$, the {\it $l$-jet space} $\cL_l(X)$  of a complex variety $X$ is the $\bC$-scheme of finite type whose set of $A$-points, for all $\bC$-algebras $A$, consists of all morphisms of $\bC$-schemes  $\Spec  A[t]/(t^{l+1})\to X$. The {\it arc space} $ \cL(X)$ is the $\bC$-scheme of infinite type whose $A$-points are the morphisms of $\bC$-schemes  $\Spec A\llbracket t\rrbracket\to X$. Truncation of arcs and jets gives a natural morphisms $$\pi_{l}: \cL(X)\to\cL_l(X),\quad
\pi_{l,l'}: \cL_l(X)\to \cL_{l'}(X),
$$  
for $l,l'\in\bN$ with $l\ge l'$. 

From now on we use the setup from the introduction: $X$ is smooth, $D$ an effective divisor, $\Sigma\neq\emptyset$ a closed subset of $D_{red}$. In this setting we have defined the $m$-contact locus $\sX (X,D,\Sigma)$ for  non-zero $m\in\bN$, see Definition \ref{defsX}. An alternative definition in terms of jet spaces of $D$ is
\be\label{eqClJs}
\sX _m(X,D,\Sigma) = \left [\pi_{ m-1}^{-1}(\cL_{m-1}(D))\setminus \pi_{m}^{-1}(\cL_m(D))\right ]
\cap \pi_{0}^{-1}(\Sigma). 
\ee
 If $D$ is reduced and smooth, then the image $\pi_m(\sX_m(X,D,\Sigma))$ in $\cL_m(X)$  is Zariski locally trival fibered by $\bC^{mn-m+1}\setminus\bC^{mn-m}$ over $\Sigma$, where $n=\dim X$. Hence  $\sX_m(X,D,\Sigma)=\pi_m^{-1}\pi_m(\sX_m(X,D,\Sigma))$ has codimension $m+\codim_D\Sigma$ and is irreducible if $\Sigma$ is.

 The main reason to use $m$-separating log resolutions in relation with contact loci is the next proposition. The version for restricted contact loci has appeared in \cite[Lemma 2.1, Lemma 2.6]{Fl}, see also \cite{ELM}. We use the notation introduced in  \ref{submN}. 

\begin{prop}\label{propCli}
Let $(X, D,\Sigma)$ be such that $X$ is a smooth complex algebraic variety, $D\neq 0$ is an effective integral divisor on $X$, and $\Sigma\neq\emptyset$ is a Zariski closed subset of the support of $D$. Let $m\in\bZ_{>0}$ and let $\mu:Y\to X$ be an $m$-separating log resolution of $(X,D,\Sigma)$. Then there is a disjoint union decomposition of the $m$-contact locus of $(X,D,\Sigma)$ as in (\ref{eqD}) such that, for all $l\in\bN$ large enough, $\sX_{m,i} $ is the inverse image under $\pi_{l}$ in the arc space $ \cL(X)$ of a non-empty smooth  irreducible Zariski locally closed subset $\sX^l_{m,i}$ of codimension $m\nu_i/N_i$ in the $l$-jet space $\cL_l(X)$. Moreover, $\sX^l_{m,i}$ is homotopy equivalent to the normal circle bundle of $E_i^\circ$.
\end{prop}
\begin{proof} The proof is similar to that of \cite[Lemma 2.1, Lemma 2.6]{Fl} for restricted contact loci. The differences in the proof are as follows. Unlike for restricted contact loci, one does not need to consider the finite unramified covers $\tilde E_i^\circ\to E_i^\circ$ from the proof of \cite[Lemma 2.6]{Fl}. Thus,  ``smooth variety" in the statement of \cite[Lemma 2.6]{Fl} means possibly finitely many disjoint copies of an irreducible smooth variety, while in our case we have irreducibility of $\sX^l_{m,i}$. Another difference is that the varieties $\sX^l_{m,i}$ have dimension one more than the their version for restricted contact loci. The extra coordinate takes any value in $\bC^*$, this being the coefficient of $t^{k_i}$ in the last line of the proof of \cite[Lemma 2.6]{Fl}. This identifies the analog of the morphism $\tilde{\pi}_0^l$ from that proof, up to homotopy, with the $\bC^*$-bundle on $E_i^\circ$ obtained by removing the zero section of the normal bundle.
\end{proof}

\noi {\bf Proof of Proposition \ref{corLct}.} It follows directly from the previous proposition.
$\hfill\Box$

\comment{

\sm

The following lemma is for clarification purposes and not used in the text:

\begin{lemma}\label{subInfi} With $(X,D,\Sigma)$ as in Proposition \ref{propCli}, there are infinitely many contact valuations for $(X,D,\Sigma)$. 
\end{lemma}
\begin{proof}
If not, then there exists $m_0\in\bZ_{>0}$ such that every contact valuation is a contact $m$-valuation for some $1\le m\le m_0$. For each $1\le m\le m_0$ consider an $m$-separating log resolution $\mu_m$ of $(X,D,\Sigma)$. Then all the contact valuations can be computed from the set $\mu_1,\ldots, \mu_{m_0}$. Take $p$ a prime number much bigger with respect to the finitely many orders of vanishing $N_i$ of $D$ along the finitely many prime divisors $E_i$ with  $i\in \cup_{m=1}^{m_0}{S_m}$, where $S_m$ is defined for $\mu_m$ as in Definition \ref{eqSm}. Then none of $\mu_1,\ldots, \mu_{m_0}$ is a $p$-separating log resolution. For a fixed $p$-separating log resolution, the set $S_p$ consists of  prime  divisors $E$ over $\Sigma$ with $N_E=p$, and at least one of these will be a contact $p$-valuation. This gives a contradiction.
\end{proof}

}

\subs{\bf Dlt valuations.} From now on we also assume that  $D$ is a reduced divisor on $X$.  If $\mu:Y\to X$ is an $m$-separating log resolution of $(X,D,\Sigma)$, we define as in the introduction $\Delta=\mu^*(D)_{red}=\sum_{i\in S}E_i.$ By Theorem \ref{thmEMM}, $(X,D)$ admits dlt modifications, and $(Y,\Delta)$ admits minimal models over $X$.

\begin{lemma}\label{lemdlv}
Any dlt valuation $v$ of $(X,D,\Sigma)$ is an $m$-valuation of $(X,D,\Sigma)$ for any $m$ divisible by $v(D)$. Moreover,  
the center of $v$ is a prime divisor on some minimal model of $(Y,\Delta)$ over $X$ for some  $m$-separating log resolution $\mu:Y\to X$ of $(X,D,\Sigma)$.
\end{lemma}
\begin{proof} Suppose $v$ has center a prime divisor $E$ on some dlt modification $(Y',\Delta')$ of $(X,D)$ such that $E$ lies over $\Sigma$. Note that the dlt modification restricts to an isomorphism $Y'\setminus \Delta'\xa{\sim}X\setminus D$ by Proposition \ref{propFc}. Take $m$ divisible by $v(D)=\ord_ED$. By sufficiently blowing up the klt locus of $(Y',\Delta')$ we obtain an $m$-separating log resolution $\mu:Y\to X$, see  \cite[Proof of Lemma 2.9]{Fl}. Thus $v$ is an $m$-valuation. By Lemma \ref{lemDltm}, $(Y',\Delta')$ is a minimal model of $(Y,\Delta)$ over $X$ if $\Delta=(\mu^*D)_{red}$.
\end{proof}

\begin{lemma}\label{lemdlE} If $\mu:Y\to X$ is an $m$-separating log resolution of $(X,D,\Sigma)$ and $E_i\subset Y$ is the center of a dlt $m$-valuation of $(X,D,\Sigma)$, then $E_i$ does not get contracted on minimal models of $(Y,\Delta)$.
\end{lemma}
\begin{proof}
Suppose that $E_i$ is the birational strict transform of $F$, a prime divisor on another $m$-separating log resolution $\mu_1:Y_1\to X$ of $(X,D,\Sigma)$ such that $F$ does not get contracted on minimal models of $(Y_1,\Delta_1=\mu_1^*(D)_{red})$ over $X$. By performing further blowups of smooth centers we can find a log resolution $\mu_2:Y_2\to X$ which factors through both $\mu$ and $\mu_1$ and is an isomorphism over $X\setminus D$. By blowing up further we can assume that $\mu_2$ is $m$-separating, by \cite[Proof of Lemma 2.9]{Fl}. Then Lemma \ref{lemNNew} implies first that the strict transform of $F$ in $Y_2$ does not get contracted on minimal models of $(Y_2,\Delta_2=\mu_2^*(D)_{red})$, and then that $E_i$ does not get contracted on minimal models of $(Y,\Delta)$.
\end{proof}

\noi {\bf Proof of Theorem \ref{thmM}.}  Only the first inclusion in (\ref{eqInc}) needs to be proven.  We fix a dlt $m$-valuation of $(X,D,\Sigma)$. By Lemma \ref{lemdlE} it  is given by a divisor $E_i$ on some $m$-separating log resolution $Y$ for some
fixed $i\in S_m$, such that $E_i$ does not get contracted on any minimal model of $(Y,\Delta)$ over $X$, with $\Delta=\mu^*(D)_{red}=\sum_{k\in S}E_k$. We want to show that the Zariski closure
$
\ol{\sX _{m,i}}
$
is an irreducible component of $\sX _m$. Suppose that it is not. Then there exists $j\in S_m\setminus\{i\}$ such that
$
\sX _{m,i}\subset \ol{\sX _{m,j}}
$
and the latter is an irreducible component of $\sX _m$, by the decomposition (\ref{eqD}).  In particular, the codimension of $\sX _{m,i}$ is strictly bigger than the codimension of $\sX _{m,j}$. This is equivalent  to
\be\label{eqIn}
a(E_i,X,D){m}/{N_i} > a(E_j,X,D){m}/{N_j}
\ee
since $a(E_k,X,D)=\nu_k-N_k$ and ${\rm codim}\; \sX _{m,k}=m\nu_k/N_k$ 
 for all $k\in S_m$ by  Proposition \ref{propCli}. 

 We fix a minimal model $(Y',\Delta',\phi,\mu')$ of $(Y,\Delta)$ over $X$ as in (\ref{eqMM2}). We have that $E_i$ is not $\phi$-exceptional. From now we will denote by $E'_k$ the image $\phi_*E_k$ for $k\in S$, and by $S'=\{k\in S\mid E'_k\neq 0\}$. Thus $\Delta'=\phi_*\Delta=(\mu')^*(D)_{red}=\sum_{k\in S'}E'_k$.

 Pick an arc $\gamma\in \sX _{m,i}$. By assumption, $\gamma$ is a limit of arcs in $\sX _{m,j}$. 
Using the curve selection lemma, see \cite[Proof of Lemma 2.5]{Fl}, one can assume that there exists a complex analytic surface germ
$$
\al: (\bC^2,(0,0)) \to (X,\gamma(0)),\quad (t,s)\mapsto \al(t,s)
$$
such that $\al_0(t):=\al(t,0)=\gamma(t)$, and $\al_s(t):=\al(t,s)$ is an arc lifting through $E_j^\circ$ for all $s\neq 0$. By assumption, the induced meromorphic map $\tilde\al:\bC^2\dashrightarrow Y$ has 2-dimensional image and has indeterminacy. Since $E_i$ is not $\phi$-exceptional, the composition $\tilde\al'=\phi\circ\tilde\al:\bC^2\dashrightarrow Y'$ also has 2-dimensional image and has indeterminacy. Using a sequence of blowups, one resolves simultaneously the indeterminacy of $\tilde\al$ and $\tilde\al'$ to obtain a commutative diagram
$$
\xymatrix{
Z \ar[rd]_{\beta}  \ar[rrrd]^{\beta'} \ar[dd]_\sigma & & &\\
& (Y, \Delta) \ar[rd]^<<<<{\mu} \ar@{-->}[rr]^{\phi\quad} & &(Y', \phi_*\Delta) \ar[dl]^{\mu'}\\
(\bC^2,0) \ar@{-->}[ru]^{\tilde\al}  \ar@{-->}[rrru]_<<<<<<<<<<<<<<<{\tilde\al'} \ar[rr]_{\al}& & X .&
}
$$
Here we work with a small neighborhood of the origin in $\bC^2$.

We write $L_s$ for the line $\{(t,s)\mid t\in\bC\}$ in $\bC^2$. Denote by $\sigma^{-1}(0)=\cup_{a\in A}F_a$ the irreducible components of the exceptional divisor of $\sigma$. Then we have rationally equivalent cycles
$$
\sigma^*L_s=\tilde{L}_s\; (\text{for }s\neq 0)\quad \text{and}\quad \sigma^*L_0 = \tilde{L}_0+\sum_{a\in A}b_aF_a,
$$
for some integers $b_a> 0$ and $F_a$ $\sigma$-exceptional prime divisors, where $\tilde{L}_s$ denotes the strict transform.

Let $G$ be the strict transform under $\sigma$ of the $s$-axis in $\bC^2$. Then $\beta(G)\subset E_j$ since $E_j$ is proper and $\tilde\al_s(0)\in E_j^\circ$ for generic small $s \neq 0$.  The inverse image in $Z$  under $\beta$ of any subset of the support of $\Delta$ is contained in $(\cup_aF_a)\cup G$. This is because $\al^{-1}(D)$ is contained in the $s$-axis, shrinking to a smaller neighborhood of $0\in\bC^2$ if necessary, and $Z$ factors through $\bC^2\times_XY$. Similarly for $\beta'^{-1}$ of subsets of the support of $\phi_*\Delta$.

Consider the $\bQ$-divisor on $Y'$
$$
W :=\sum_{k\in S'} a(E_k',X,D)E'_k= \sum_{k\in S'} a(E_k,X,D)E'_k.
$$
Then $W$ is $\mu'$-exceptional and there is a $\bQ$-linear equivalence of $\bQ$-Cartier $\bQ$-divisors
$$
W = K_{Y'}+\Delta' -\mu'^*(K_X+D)
$$
by the definition of log discrepancy.
As above, $\beta'^*W$ is supported on $(\cup_aF_a)\cup G$. 
The intersection products 
\be\label{eqeq}\sigma^*L_s\cdot \beta'^*W = \sigma^*L_0\cdot \beta'^*W\ee
are well-defined and are equal for all small $s$, since the same is true for the intersection with the compact $F_a$, while for the non-compact $G$, $\sigma^*L_s\cdot G=1$ for all $s$ by the projection formula.
We will derive a contradiction from computing both sides. For the right-hand side we use the  nefness property of $K_{Y'}+\Delta'$. For the left-hand side we use the property that log discrepancies do not decrease on relative minimal models. 

The right-hand side is
$$
\sigma^*L_0\cdot \beta'^*W = \beta'_*\sigma^*L_0\cdot W = \beta'_*\tilde L_0 \cdot W + \sum_{a\in A}b_a\beta'_*F_a\cdot W
$$
Since $E_i$ is not $\phi$-exceptional, $\beta'_*\tilde L_0\cdot E_k' =0$ if $k\neq i$ and $\beta'_*\tilde L_0\cdot E_i' =\ord_{\gamma}(E_i)=m/N_i$. Thus $\beta'_*\tilde L_0 \cdot W=a(E_i,X,D){m}/{N_i}$.  Further, since $\beta'_*F_a$ is contained in the $\mu'$-exceptional locus of $Y'$, $\beta'_*F_a\cdot \mu'^*(K_X+D)=0$. Thus  
$\beta'_*F_a\cdot W=\beta'_*F_a\cdot (K_{Y'}+\Delta')$ and this is $\ge 0$ by the $\mu'$-nefness of $K_{Y'}+\Delta'$. Since $b_a\ge 0$, we have obtained that
\be\label{eqeq1}
\sigma^*L_0\cdot \beta'^*W \ge  a(E_i,X,D){m}/{N_i}.
\ee

On the other hand, for $s\neq 0$,
$$
\sigma^*L_s\cdot \beta'^*W = \tilde L_s\cdot \beta'^*W.
$$
Note that in general it is not necessarily true for divisors that $\beta^*\circ \phi^*=\beta'^*$ since $\phi$ is only a birational map and not a morphism, see \cite[1.20]{K}. However, $\phi$ is a morphism in codimension 1. Hence $\phi$  is a morphism on a neighborhood in $Y$ of a general point of $E_j$. We can assume the center of the arc $\tilde\al_s$ is such a general point of $E_j^\circ$. Then 
$$\tilde L_s\cdot \beta'^*W = \beta_*L_s\cdot\phi^*W.$$
 The only contribution to this product is from $E_j$ and its coefficient in $\phi^*W$. We can estimate it as follows. One has by the definition of log discrepancy that
$$
K_Y+\Delta = \phi^*(K_{Y'}+\Delta') + \sum_{k\in S\setminus S'} a(E_k,Y',\Delta') E_k
$$
where the sum is over the $\phi$-exceptional divisors in $Y$, and
$$
K_Y+\Delta  =\mu^*(K_X+D) +\sum_{k\in S}a(E_k,X,D)E_k.
$$
Note that on divisors $\phi^*\circ\mu'^*=\mu^*$ since  $\mu'$ is a morphism, see \cite[1.20.1]{K}, hence $\mu^*(K_X+D)=\phi^*\mu'^*(K_X+D)$.
Thus
$$
\phi^*W = \sum_{k\in S} (a(E_k,X,D) - a(E_k,Y',\Delta'))E_k,
$$
since $a(E_k,Y',\Delta')=0$ if $k\in S'$. Hence
$$
\tilde L_s\cdot \beta'^*W = (a(E_j,X,D) - a(E_j,Y',\Delta')){m}/{N_j}.
$$
By the definition of minimal models of $(Y,\Delta)$ over $X$, 
$
a(E_j,Y',\Delta')\ge a(E_j,Y,\Delta) =0.
$
We have thus obtained that
$$
\sigma^*L_s\cdot \beta'^*W\le a(E_j,X,D){m}/{N_j}.
$$
From (\ref{eqeq}) and (\ref{eqeq1}), this implies 
$
a(E_j,X,D){m}/{N_j} \ge a(E_i,X,D){m}/{N_i},
$
which contradicts (\ref{eqIn}). 
$\hfill\Box$

\medskip

We note that the hypothesis that $X$ is smooth was used in (\ref{eqIn}) and that $D$ is reduced was used in the analysis of $W$.

\subs{\bf Jump valuations.} Now we prove Theorem \ref{thmJV}. First recall the following terminology.

\begin{defn}\label{defMJ} Let $X$ be smooth complex algebraic variety and $D$ an effective divisor on $X$.
 Let $\mu:Y\to X$ be a log resolution of $(X,D)$. For $\lam\in \bQ_{>0}$, the {\it multiplier ideal sheaf of $(X,\lam D)$} is
$$
\cJ(X,\lam D) :=\mu_*\cO_{Y}(K_{Y/X}-\lfloor \lam\mu^*D\rfloor) \subset \cO_X,
$$
where $\lfloor\_\rfloor$ means taking the round-down of each coefficient in a $\bQ$-divisor, and $K_{Y/X}$ is the unique $\mu$-exceptional divisor linearly equivalent to $K_Y-\mu^*K_X$.  It is well-known that this definition does not depend on $\mu$.
A number $\lam\in \bQ_{>0}$ is called a {\it jumping number of $(X,D)$} if $\cJ(X,\lam D)\subsetneq \cJ(X,(\lam-\eps)D)$  for all $0<\eps\ll 1$. We say that a prime $\mu$-exceptional divisor $E_i$ on $Y$ {\it contributes with $\lam$ as a jumping number of $(X,D)$} if $\lam=(\nu_i+b)/N_i$ for some $b\in\bN$ and
$$
\cJ(X,\lam D)\subsetneq \mu_*\cO_Y(K_{Y/X}-\lfloor \lam\mu^*D\rfloor +E_i).
$$ 
This  depends on the valuation defined by $E_i$ and not on $\mu$, and implies that $\lam$ is a jumping number.
\end{defn}

\noi{\bf Proof of Theorem \ref{thmJV}.} {\it Step 1: Preparation.} We can assume that $E_i$ is $\mu$-exceptional since otherwise the claim is true. It is enough to show that $E_i$ does not get contracted on a minimal model of $(Y,\Delta)$ over $X$, where $\Delta=\mu_*^{-1}D + \Ex(\mu)$. Hence it is enough to show that $E_i$ does not get contracted after any step of the minimal model program for $(Y,\Delta)$ over $X$. Here by a step of the minimal model program we mean an extremal-ray contraction followed by a flip if necessary. Let $\mu':(Y',\Delta')\to X$ be obtained after a succession of steps, with $\Delta'=\phi_*\Delta={\mu'_*}^{-1}D+\Ex(\mu')$ and $\phi:Y \dashrightarrow Y'$ is the natural birational map. Assume that $(Y',\Delta')$ is the last step before $E_i$ gets contracted. Thus $E_i'=\phi_*(E_i)\neq 0$ contains an open dense subset covered by curves of numerical class $C$ such that $(K_{Y'}+\Delta')\cdot C< 0$. We will derive a contradiction.

We do not need $Y$ anymore, but we need some log resolution of $(X,D,\Sigma)$ which factors through $\mu'$ and on which the valuation defined by $E_i$ has a prime center, to compare contributions to jumping numbers. To avoid introducing new notation, by taking a common log resolution we can  assume that $\phi:Y\to Y'$ is a birational projective morphism.

{\it Step 2: Multiplier ideals are computed by $\mu'$.} Since $(Y',\Delta')$ is a log canonical pair, the multiplier ideals $\cJ(X,\lam D)$ are computed via $\mu'$, by a result of Smith-Tucker \cite[Theorem A.2]{BV}, that is, 
\be\label{eqTS}
\cJ(X,\lam D) =\mu'_*\cO_{Y'}(K_{Y'/X}-\lfloor \lam\mu'^*D\rfloor) 
\ee
for every $\lam \in\bR_{>0}$. Here $K_{Y'/X}$ is the unique exceptional divisor $\bQ$-linearly equivalent to $K_{Y'}-\mu'^*K_X$. It has  integral coefficients since $X$ is smooth.  We recall the proof  since we will need an intermediate step. Since there is an equality of divisors
$
\phi_*(K_{Y/X}-\lfloor \lam\mu^*D\rfloor)=K_{Y'/X}-\lfloor \lam\mu'^*D\rfloor,
$
there is an inclusion of sheaves
$
\phi_*\cO_Y(K_{Y/X}-\lfloor \lam\mu^*D\rfloor)\subset\cO_{Y'}(K_{Y'/X}-\lfloor \lam\mu'^*D\rfloor).
$
We will show 
\be\label{eqTSa}
\phi_*\cO_Y(K_{Y/X}-\lfloor \lam\mu^*D\rfloor) = \cO_{Y'}(K_{Y'/X}-\lfloor \lam\mu'^*D\rfloor).
\ee
Taking direct image under $\mu'$ we then get (\ref{eqTS}). 

Let $g$ be a section of  $\cO_{Y'}(K_{Y'/X}-\lfloor \lam\mu'^*D\rfloor)$, that is, 
\be\label{eqw1}
\Div(g) +K_{Y'/X}-\lfloor \lam\mu'^*D\rfloor\ge 0.
\ee
Pick $\eps\in\bQ_{>0}$ small enough such that $\Delta -\eps\mu^*(\lam D)\ge \{\mu^*(\lam D)\}$. Here  $\{\_\}$ means taking the fractional part of each coefficient. This also implies that $\Delta' -\eps\mu'^*(\lam D)\ge \{\mu'^*(\lam D)\}$ by applying $\phi_*$. Then
$\Delta -\eps\lfloor \mu^*(\lam D)\rfloor - (1+\eps)\{\mu^*(\lam D)\}\ge 0,
$ and so
\be\label{eqw3}
\Delta' -\eps\lfloor \mu'^*(\lam D)\rfloor - (1+\eps)\{\mu'^*(\lam D)\}\ge 0.
\ee
Adding (\ref{eqw1}) and (\ref{eqw3})  we get
$
\Div(g) +  K_{Y'/X}+\Delta' - (1+\eps)\mu'^*(\lam D)\ge 0.
$
The left-hand side is $\bQ$-Cartier, hence we can pull it back via $\phi$. We obtain
$$
\Div(g\circ\phi) -  \mu^*K_X  +\phi^*(K_{Y'}+\Delta')- (1+\eps)\mu^*(\lam D)\ge 0.
$$
Since $(Y',\Delta')$ has log canonical singularities, $K_{Y}+\Delta\ge \phi^*(K_{Y'}+\Delta')$, and hence
$$
\Div(g\circ\phi) -  \mu^*K_X  +K_{Y}+\Delta- (1+\eps)\mu^*(\lam D)\ge 0.
$$
Taking integral part it follows that
$
\Div(g\circ\phi) +  K_{Y/X}-\lfloor \mu^*(\lam D)\rfloor\ge 0,
$
that is, $g$ is a section of $\phi_*\cO_Y(K_{Y/X}-\lfloor \lam\mu^*D\rfloor) $. This proves (\ref{eqTSa}).

{\it Step 3: Contribution to jumping numbers can be read from $\mu'$.} Assume now that $E_i$ contributes with the jumping number $\lam$ of $(X,D)$. We show that $E_i'$ also does, that is,
\be\label{eqTS1}
\cJ(X,\lam D) \subsetneq \mu'_*\cO_{Y}(K_{Y'/X}-\lfloor \lam\mu'^*D\rfloor +E_i').
\ee
Since there is an equality of divisors
$
\phi_*(K_{Y/X}-\lfloor \lam\mu^*D\rfloor+E_i)=K_{Y'/X}-\lfloor \lam\mu'^*D\rfloor +E_i',
$
there is an inclusion of sheaves
$
\phi_*\cO_Y(K_{Y/X}-\lfloor \lam\mu^*D\rfloor+E_i)\subset\cO_{Y'}(K_{Y'/X}-\lfloor \lam\mu'^*D\rfloor+E_i').
$
Taking direct image under $\mu'$ we obtain
$
\mu_*\cO_Y(K_{Y/X}-\lfloor \lam\mu^*D\rfloor +E_i) \subset \mu'_*\cO_{Y'}(K_{Y'/X}-\lfloor \lam\mu'^*D\rfloor +E_i').
$
By assumption, the first sheaf strictly contains $\cJ(X,\lam D)$. Then (\ref{eqTS1}) follows.

{\it Step 4: Local vanishing for $\mu'$.} 
Since $\mu=\mu'\circ\phi$, the Grothendieck spectral sequence implies that there is a injective morphism of sheaves of $\cO_X$-modules
$$R^1\mu'_*(\phi_*\cO_Y(K_{Y/X}-\lfloor \lam\mu^*D\rfloor))\subset R^1\mu_* \cO_Y(K_{Y/X}-\lfloor \lam\mu^*D\rfloor).$$ 
By local vanishing \cite[Theorem 1.2.3]{KMM}, the last sheaf is $0$, hence the first sheaf is also $0$. By (\ref{eqTSa}) this means that
\be\label{eqLv}
R^1\mu'_*\cO_{Y'}(K_{Y'/X}-\lfloor \lam\mu'^*D\rfloor) =0.
\ee

{\it Step 5: Non-vanishing.}  There is a short exact sequence of sheaves of $\cO_{Y'}$-modules
$$
0\to \cO_{Y'}(K_{Y'/X}-\lfloor \lam\mu'^*D\rfloor) \to \cO_{Y'}(K_{Y'/X}-\lfloor \lam\mu'^*D\rfloor +E_i')\to \cQ\to 0
$$
for some sheaf $\cQ$. 
Applying $\mu_*'$ and using (\ref{eqTS}), (\ref{eqTS1}), and (\ref{eqLv}), we get  a short exact sequence of sheaves of $\cO_X$-modules
$$
0\to \cJ(X,\lam D) \to \mu'_*\cO_{Y'}(K_{Y'/X}-\lfloor \lam\mu'^*D\rfloor +E_i')\to \mu'_*\cQ \to 0
$$
such that the last term is non-zero. Assuming that $X$ is affine, which we can since the problem is local on $X$, this implies that 
$
H^0(Y',\cQ)\neq 0
$
and 
$$
H^0(Y',\cO_{Y'}(K_{Y'/X}-\lfloor \lam\mu'^*D\rfloor)) \subsetneq H^0(Y',\cO_{Y'}(K_{Y'/X}-\lfloor \lam\mu'^*D\rfloor +E_i')).
$$
Thus there exists a rational function $g$ on $Y'$ such that $$P=\Div(g)+K_{Y'/X}-\lfloor \lam\mu'^*D\rfloor+E_i'\ge 0$$ but $P-E_i'$ is not effective. Therefore $E_i'$ has coefficient $0$ in $P$, and $P$ is linearly equivalent to $K_{Y'/X}-\lfloor \lam\mu'^*D\rfloor+E_i'$.

{\it Step 6: The contradiction.} We have that $K_{Y'}+\Delta' $ is $\bQ$-linearly equivalent over $X$ to
$$
K_{Y'/X}-\mu'^*(\lam D) +\Delta' = (K_{Y'/X}-\lfloor \mu'^*(\lam D)\rfloor +E_i') + (\Delta' -E_i'-\{\mu'^*(\lam D)\}).
$$
 Note that $R=\Delta' -E_i'-\{\mu'^*(\lam D)\}$ is effective and its support does not contain $E_i'$. By Step 1, we have
$
0>C\cdot (K_{Y'}+\Delta')$.
Using the splitting from above and $\bQ$-factoriality of $Y'$, we get
$0>C\cdot (P+R)|_{E_i'},$
where $P$ is as in Step 5. Note that the restriction of $P+R$ to $E_i'$ is a well-defined effective $\bQ$-Cartier $\bQ$-divisor. But the curves in $E_i'$ in the numerical equivalence class $C$ cover an open subset of $E_i'$, hence $C\cdot (P+R)|_{E_i'}\ge 0$. This is a contradiction. 
$\hfill\Box$

\subs{\bf Proof of Theorem \ref{thmLcv}.}
Let $\mu:Y\to X$ be an $m$-separating log resolution of $(X,D,\Sigma)$. Let $\Delta=(\mu^*D)_{red}$. Since $(X,D)$ is log canonical, one can apply \cite[1.35]{K} with $c_i=1$. The conclusion is that $(Y,\Delta)$ has a minimal model over $X$ that contracts all  $\mu$-exceptional prime divisors $E_i$ with $\nu_i/N_i\neq 1$ and it contracts no other prime divisors.  This implies (i) and (ii). 

If $\Sigma=D$, then $\lct_m(X,D,\Sigma)=1$ for all $m\ge 1$ and (i) applies. Moreover, since $(X,D)$ is log canonical, then $D$ is log canonical, by \cite[7.3.2]{Ko}. Since $D$ is log canonical, the $m$-th jet scheme $\cL_m(D)$ is equidimensional for every $m$, by \cite[Theorem 1.3]{EM}. By definition, $\sX^m_m(X,D)$, the $m$-contact locus in $\cL_m(X)$, is a Zariski open subset of $\pi_{m,m-1}^{-1}(\cL_{m-1}(D))$, where $\pi_{m,m-1}:\cL_m(X)\to \cL_{m-1}(X)$ is the truncation morphism. Since the latter is a trivial fibration, $\sX^m_m(X,D)$ is also equidimensional. Since $\sX _m(X,D)=\pi_{ m}^{-1}(\sX_m^m(X,D))$, it follows that all irreducible components of $\sX _m(X,D)$ have the same codimension. This proves (iii). 

One knows that $D$ has rational singularities if and only if $(X,D)$ is dlt, by applying 
 \cite[(7.9), (11.1.1)]{Ko} to our setup, namely, $X$ is smooth and $D$ is reduced, and taking $Z=Sing(D)$ in Definition \ref{defDltY}. In this is the case, $\cL_m(D)$ is irreducible for all $m\ge 1$ by \cite[Theorem 0.1]{M1}, and thus $\sX _m(X,D)$ is also irreducible by the same argument as above. This proves (iv).
$\hfill\Box$

\section{Hyperplane arrangements}\label{secHA}

In this section we prove Theorem \ref{corArr}. We consider pairs $(X=\bC^n,D)$ where  $D$ is a non-zero  effective divisor supported on a union of hyperplanes. We call $D$ a {\it hyperplane (multi-)arrangement} if it is (maybe non-)reduced.  The {\it canonical log resolution} of $(X,D)$ is obtained by blowing up successively by increasing dimension the  strict transforms of edges of $D$. An {\it edge} is an intersection of some hyperplanes in the arrangement. 

\begin{rmk} A minimal (that is, factoring all others) log resolution exits.  In contrast, there is  no minimal $m$-separating log resolution for $(X,D)$, e.g. $D=\{xyz=0\}\subset\bC^3$ and $m=2$.
\end{rmk}

\begin{defn}
A log resolution of $(X,D)$ is {\it good} if it is obtained by successively blowing up, starting from the canonical log resolution, non-empty intersections of two distinct irreducible components of the strict transform of $D$. Such resolutions exist by \cite[Proof of Lemma 2.9]{Fl}. 
\end{defn}

The proof of Theorem \ref{corArr} rests on two  results. The first one was obtained in \cite[Proposition 7.4]{BT}, where the statement, but not the proof, needs ``good" to be added:

\begin{prop}\label{propTue}(\cite{BT})
If  $D$ is a hyperplane multi-arrangement in $X=\bC^n$, and $m\in\bZ_{>0}$, then  (\ref{eqD}) is the decomposition into connected components of the  $m$-contact locus $\mathscr{X}_m (X,D)$, for any good $m$-separating log resolution of $(X,D)$.\end{prop}


\begin{prop}\label{propDltArr} Let $(X=\bC^n,D)$ be a hyperplane arrangement. Let $\mu:Y\to X$ be a good log resolution. Let $\Delta=\mu_*^{-1}D+\Ex(\mu)$. Then $(Y,\Delta)$ is its own minimal model over $X$.
\end{prop}
\begin{proof} It is enough to show that the $K_Y+\Delta$ is $\mu$-nef. In fact, it is enough to prove that $K_{\bar Y}+\bar\Delta$ is nef, where $\bar\mu:\bar Y\to \bar X$ is a good log resolution of $(\bar X, \bar D+H)$, $\bar X=\bP^n=\bC^n\sqcup H$, $H$ is the hyperplane at infinity, $\bar D$ is the compactification of $D$ in $\bP^n$, and $\Delta$ is $(\bar\mu^{-1}(\bar D + H))_{red}$. Since $\bar\mu$ is a good log resolution, $\bar Y$ is the compactification of $X\setminus D$ in a toric variety that maps under a proper toric morphism to $\bP^r$, where $r$ is the number of hyperplanes in $\bar D+ H$, see \cite[\S 4]{Te}. Then, by \cite[Theorem 1.4]{Te} (and the 2nd to last paragraph of its proof) and \cite[Theorem 1.5]{Te}, $K_{\bar Y}+\bar \Delta$ is the pullback under a proper morphism of a very ample divisor. In particular, $K_{\bar Y}+\bar \Delta$ is nef.
\end{proof}

We note that the proof show moreover that $(Y,\Delta)$ is a good minimal model over $X$.

\medskip

\noi{\bf Proof of Theorem \ref{corArr}.}
Since the subsets $\sX _{m,i}$ are irreducible, Proposition \ref{propTue} implies that the set of contact $m$-valuations is the set of essential $m$-valuations, and both consist of the divisorial valuations corresponding to $E_i$ with $i\in S_m$ as in Definition \ref{eqSm} for any good $m$-separating log resolution.  By Proposition \ref{propDltArr}, every such $m$-valuation is a dlt valuation of $(X,D)$. 
$\hfill\Box$

\section{Examples with rational singularities}\label{secOth}

If a hypersurface $D$ in a smooth variety $X$ has rational singularities, there are no dlt valuations for $(X,D,D_{sing})$ by Theorem \ref{thmLcv}. We look at some examples.


\comment{

\subs{\bf Rational double point surface singularities.}\label{subsRDP} Let $O\in\bC^3$ be the origin and $(D,O)$ one of the rational double point singularities: 
\begin{center}
$A_n$ $(n\in\bN)$, $D_n$ $(n\in\bN, n\ge 4)$, $E_6$, $E_7$, $E_8$. 
\end{center}
In \cite{Mou-rat}  the irreducible components of the  spaces of jets of $D$ passing through $O$ are determined. Using the relation (\ref{eqClJs}) between jet spaces and contact loci, this determines the irreducible components of $\sX _m(\bC^3,D,O)$   $(m\in\bN)$. All of them have codimension $m+1$ for $m\ge N-1$, where $$N = n+1, 2n-2, 12, 18, 30,$$ respectively. Hence every contact $m$-valuation is a top contact $m$-valuation. This also implies that $\lct_N(\bC^3,D,O)={\codim\;\sX _N(\bC^3,D,O)}\cdot 1/{N}=1+{1}/{N}.$ This  coincides with the minimal exponent  $\al_D$, by using the identification of $\al_D$ in the isolated hypersurface case with the minimal spectral number of $f$ and consulting the table from \cite[p. 129]{Kul}.  Thus Question \ref{queAf} has a positive answer in this case.

}

\subs{\bf Hypersurfaces with isolated maximal-multiplicity rational singularities.}\label{subImlr} Let $D$ be a hypersurface in $\bC^n$ with at most rational singularities at the origin $O$, which is an isolated singularity, and such that   $D$ has multiplicity  $n-1$ at $O$. In this case it follows from the proof of \cite[Proposition 3.7]{BMS}, see also \cite{Mou-rat} for $n=3$,  that for $m\ge n-1$, $\sX_m (\bC^n,D,O)$ has codimension $m+1$, and all its irreducible components have the same codimension. Hence, for such $m$, all contact $m$-valuations are top contact $m$-valuations, and 
$
\lct_m(\bC^n,D,O)=1+{1}/{m}.
$

\subs{\bf Determinants of  generic square matrices.}\label{subDet} Let $f=\det (x_{ij})$, where $(x_{ij})_{1\le i,j\le n}$ is a matrix of independent variables. It defines a hypersurface $D$ in $M=\bC^{n^2}$ with rational singularities. The jet spaces of $D$ together with their stratifications in terms of orbits of a certain group action have been determined by Docampo \cite{Roi}. Using (\ref{eqClJs})  this allows to conclude that $\sX _m(M,D,D_{sing})$ is irreducible of codimension $m+2$ for $m\ge 1$. To prove this one notes that $D=D^{n-1}$, $D_{sing}=D^{n-2}$, where $D^k$ is the subvariety of matrices in $M$ with rank at most $k$. Then, in the notation from \cite{Roi},
$
\sX _m(M,D,D_{sing}) = \ol{\cC_\lam}\setminus \ol{\cC_{\mu}},
$
where: $\lam, \mu\in\ol{\Lambda}_n$ are the pre-partitions $(m-1,1,0,\ldots,0)$, $(m,1,0,\ldots,0)$, respectively; $\cC_\lam, \cC_\mu$ are the respective orbit closures of $ \cL(GL_n(\bC))^{\times 2}$ on $ \cL(M)$; and $\ol{(\_)}$ denotes the Zariski closure. Then \cite[Prop. 5.4]{Roi} gives the codimension. There are no dlt $m$-valuations by Theorem \ref{thmLcv}, and the only contact $m$-valuation is thus a top $m$-valuation. It also follows that
$
\lct_m(M,D,D_{sing}) = {\codim\;\sX _m(M,D,D_{sing})}\cdot 1/{m}= {(m+2)}/{m}.
$
\comment{
It is well-known that the $b$-function of $f$ is 
$b_f(s)=(s+1)(s+2)\ldots (s+n).$
Thus the minimal exponent of $f$ is
$
\al_f = 2 =\lct_2(M,D,D_{sing}).
$
Hence Question \ref{queAf} is true in this case.}


\subs{\bf Homogeneous polynomials with an isolated singularity.}\label{subIhp}
\label{ex2}  Let $X=\bC^n$, $n\ge 3$, and let $D$ be given by a homogeneous polynomial of degree $d$ with an isolated singularity at the origin $O$.  The blowup $\mu:Y\to X$ at $O$ is a minimal $d$-separating log resolution of $(X,D,O)$. Then $\Delta=(\mu^*D)_{red}=\tilde D +E$, where $E\simeq\bP^{n-1}$ is the  exceptional divisor and $\tilde D$ is the strict transform of $D$. Then $E$ gives the only essential  $d$-valuation of $(X,D,O)$. By (\ref{eqD}), this is also the only contact $d$-valuation of $(X,D,O)$. If $d<n$  there are no dlt $d$-valuations for $(X,D,O)$, and if $d\ge n$ then $E$ gives the only dlt $d$-valuation. Indeed, we have $K_Y+\Delta=\mu^*(K_X+D)+ (n-d)E$, where $a(E,X,D)=n-d$. If $d<n$ then $(X,D)$ is the relative minimal model  of $(Y,\Delta)$ by Proposition \ref{propFc}, and $\lct_d(X,D,O)={n}/{d}$.
 If $d\ge n$, then $\bB(K_Y+\Delta/X)=\bB((n-d)E/X)=\emptyset$, since $-E$ is ample over $X$. Thus $(Y,\Delta)$ is its own  minimal model over $X$ by Proposition \ref{propBE}.



\section{Curves on surfaces: essential and dlt valuations}\label{secCur}

Let $X$ be a smooth complex algebraic surface. We give the geometric characterization of the essential valuations and of the dlt valuations in terms of the resolution graph of $(X,D,\Sigma)$. 

\subs{\bf Essential valuations.}\label{subEss}

\begin{prop}\label{propMSL} Let $X$ be a smooth complex algebraic surface, $D$ a non-zero effective divisor on $X$, and $\Sigma\neq\emptyset$ a Zariski closed subset of the support of $D$. Let $m>0$ be an integer. Then a minimal (that is, factoring all other) $m$-separating log resolution of $(X,D,\Sigma)$ exists.
\end{prop}
\begin{proof} 
Let $\mu:Y\to X$ be the minimal log resolution of $(X,D,\Sigma)$. If $E_i$ and $E_j$ denote two distinct irreducible components of $\mu^{-1}(\Supp(D))$ intersecting in a point $P$ and such that $N_i+N_j\le m$, then blowup the point $P$. The multiplicity of the new exceptional divisor is $N_i+N_j$. Since $N_i, N_j$ are strictly positive, $N_i+N_j>\max\{N_i,N_j\}$. We repeat this process until we obtain for the first time an $m$-separating log resolution $\mu_0:Y_0\to X$. To show that the result is a minimal $m$-separating log resolution, take $\mu_1:Y_1\to X$ another $m$-separating log resolution. Let $\pi:Y_1\to Y$ be the morphism through which $\mu_1$ factors. Consider $\pi^{-1}(P)$. This is necessarily a divisor in $Y_1$, connecting the strict transforms of $E_i$ and $E_j$, since $\mu_1$ is $m$-separating. Since $\pi$ is a succession of blowups at smooth points and since $\pi^{-1}(P)$ is an exceptional divisor, it follows that $\pi$ factors through the blowup of $P$. Repeating the argument, it follows that $\mu_1$ factors through $\mu_0$.
\end{proof}

\begin{nota} {\it (Resolution graphs.)} We denote by $\Gamma$ the resolution graph of $(X,D, \Sigma)$, see \cite{BK}. Recall that $\Gamma$ is the dual graph of the divisor $(\mu^*D)_{red}$, where $\mu$ is the minimal log resolution of $(X,D,\Sigma)$, together with certain data at each vertex which remembers the resolution process, and in particular, the orders of vanishing $N_i$ of each $E_i$ as in Definition \ref{eqSm}.

For $m\ge 1$ we denote by $\Gamma_m$ the  resolution graph obtained from the minimal $m$-separating log resolution of $(X,D,\Sigma)$. 

We say that a resolution graph {\it refines} another resolution graph if it is obtained only by inserting inductively vertices $E_k$ on edges $E_iE_j$ with $i\neq j$ and such that $N_k=N_i+N_j$.

 The {\it valence} of a vertex in a graph is the number of edges connected to the vertex. A {\it path} in a graph is a connected sequence of edges with distinct end points.  We say that {\it a vertex lies on a path} if it is any of the vertices on the path, end points included. A {\it simple path} is a path that does not pass twice through the same vertex.
\end{nota}

Proposition \ref{propMSL} implies the following elementary statement, see also Lemma \ref{blow-up-formula} below:

\begin{prop}\label{propDltCu}
With the assumptions as in Proposition \ref{propMSL}:

(i) $\Gamma=\Gamma_1$ and for $m>1$, $\Gamma_m$ is  obtained inductively from $\Gamma_{m-1}$ by inserting a vertex $E_k$ with multiplicity $N_k=N_i+N_j$ on each edge $E_iE_j$ for which $N_i+N_j\le m$. Thus $\Gamma_m$ refines $\Gamma$.

(ii) For each edge $E_iE_j$ of $\Gamma$, the set of vertices $E_k$  lying on the simple path $E_iE_j$ in $\Gamma_m$ for some $m\ge 1$ is in bijection with
$
\{(a,b)\in\bN^2\mid \gcd(a,b)=1\}.
$
For such a vertex $E_k$, $N_k=aN_i+bN_j$.

(iii) The essential $m$-valuations of $(X,D,\Sigma)$ are given by the subset $S_m$ (from Definition \ref{eqSm}) of vertices of $\Gamma_m$  for the the minimal $m$-separating log resolution of $(X,D,\Sigma)$, that is,  divisors $E_k$ lying over $\Sigma$ in the minimal $m$-separating log resolution of $(X,D,\Sigma)$ such that $N_k$ divides $m$. 
\end{prop}
If $(C,0)$ is a smooth plane curve germ, $\sX_m(\bC^2,C,0)=\ol{\sX_{m,i}}$ where $i\in S_m$ is the  vertex closest to the strict transform of $C$, since $\sX_m(\bC^2,C,0)$ is irreducible of codimension $m+1$.

\subs{\bf Dlt valuations.}\label{subsDltCurves} 

\begin{prop}\label{propDltCurves} Let $X$ be a smooth complex algebraic surface, $D$ a non-zero reduced effective divisor on $X$, and $\Sigma\neq\emptyset$ a Zariski closed subset of the support of $D$. Let $m>0$ be an integer. The dlt $m$-valuations of $(X,D,\Sigma)$ are given by the vertices $E_k\in S_m\subset \Gamma_m$ that lie on the refinement of a simple path $E_iE_j$ of $\Gamma$ connecting two different vertices $E_i$, $E_j$ which are strict transforms of irreducible components of $D$ or have valence $\ge 3$ in $\Gamma$. 
\end{prop}

The rest of the section is dedicated to the proof. The proposition will follow from Corollary \ref{corSM}.

In the case of curves on smooth surfaces we have uniqueness of relative minimal models, see  \cite[Theorem 3.3 and Proposition 3.9]{Fu}:

\begin{prop}\label{propU2} Let $\mu:Y\to X$ be a projective birational morphism of surfaces, such that $Y$ is normal. Let $\Delta$ be a  boundary divisor on $Y$. Assume that $Y$ is $\bQ$-factorial, or $(Y,\Delta)$ is log canonical. Then an MMP for $(Y,\Delta)$ over $X$ terminates and gives a minimal model  $\phi:(Y,\Delta)\dashrightarrow (Y',\Delta')$ over $X$. Such minimal model is unique up to isomorphisms over $X$. Moreover, $\phi$ is a morphism and is a composition of divisorial contractions. 
\end{prop}

After we introduce some prerequisites, we characterize which components of $\Delta$  do not get contracted on the relative minimal model in Proposition \ref{propTwig}.

\begin{defn}\label{defTw}	
Let $B$ be a reduced Weil divisor on a normal surface $Z$. 

(a) If $C$ is an irreducible component of $B$, we set $\beta_{B}(C):=C\cdot (B-C).$ 

(b) If $C$ is an irreducible component of $B$, we say that $C$ is a \emph{tip} of $B$ if $\beta_{B}(C)\leq 1$. We say that a tip $C$ is \emph{admissible} if $C\simeq \mathbb{P}^1$ and $C^2<0$. 

(c)
Let $T\neq 0$ be a subchain of $B$, with irreducible components $T_1,\dots, T_{n}$, ordered in such a way that $T_{i}\cdot T_{i+1}=1$, $T_{i}\cdot T_{j}=0$ for $|i-j|>1$. Then $T$  is called a \emph{twig} of $B$ if $\beta_{B}(T_{1})\leq 1$ and $\beta_{B}(T_{i})\leq 2$ for all $i\in \{2,\dots, n\}$.  Note that the ordering of the $T_i$ is unique if the twig $T$ is not a connected component of $B$, or if $T=T_1$ and $\beta_B(T)=0$; in the remaining case the reverse ordering is the only other possible ordering. A twig $T$ is \emph{admissible} if $T_{i}\simeq \mathbb{P}^{1}$ for all $i$, and the intersection matrix $(T_{i}\cdot T_{j})_{1\le i,j\le n}$ of $T$ is negative definite. In this case, $T_{i}$ is an admissible tip of $B-\sum_{j=1}^{i-1}T_{j}$, $i\in \{1,\dots n\}$. An (admissible) tip of $B$ is thus by definition an (admissible) twig of $B$, corresponding to the case $n=1$.

(d)  If $Z\to Z'$ is a birational morphism to another normal surface, and $B'$ is a reduced divisor on $Z'$, we say that $T$ is an (admissible) twig of $B$ \emph{over $B'$} if it is an (admissible) twig of $B$ and has no common irreducible component with the strict transform of $B'$.

(e) If $T\subset B$ is a subdivisor, $d(T)$ will denote the determinant of the negative of the intersection matrix of $T$, and we set $d(0)=1$.
\end{defn}

\begin{prop}\label{propCQ} Let $B$ be a reduced  divisor on a smooth surface $Z$. 

(a) If $T$ is an admissible twig of $B$  then $T$ can be contracted to a cyclic quotient singularity $Z'$.

(b) If in addition $T=\sum_{i=1}^n T_i$ as in Definition \ref{defTw}, then the log discrepancy of $T_i$ on  $(Z',B')$, where $B'$ is the divisorial image  of $B$, equals $ d(T-\sum_{j=1}^{i}T_{j})/d(T)$ if $T$ is not a connected component of $B$, and $(d(T-\sum_{j=i}^{n}T_{j})+d(T-\sum_{j=1}^{i}T_{j}))/d(T)$ otherwise.

\end{prop}	
\begin{proof}
(a) This is Grauert's Contraction Theorem \cite[Theorem III.2.1]{BPV} together with the standard description of Hirzebruch-Jung singularities, see e.g. \cite[Theorem 3.32]{K} or  \cite[III. 2.(ii) and Theorem III.5.1]{BPV}. Note that this description is formulated for chains with no $(-1)$-curves, that is for minimal resolutions of cyclic singularities. Starting from an arbitrary negative definite chain of $\bP^1$'s, contract all $(-1)$-curves in this chain and its images. Then the image of the chain is either a smooth point, or a chain as in the cited theorems.

(b) This is proved in \cite[3.1.10]{K-ab}. One applies the formula in the cited theorem to the three possible cases concerning the ordering of the $T_i$ mentioned in Definition \ref{defTw}. For the two cases when $T$ is a connected component, the outcome reads the same.
\end{proof}

\begin{prop}\label{propTwig}  Let $X$ be a smooth complex algebraic surface and $D$ a non-zero reduced divisor on $X$. Let $\mu:Y\to X$ be a log resolution of $(X,D)$ that is an isomorphism over $X\setminus D$. Let $\Delta=(\mu^*D)_{red}$. Let  $\phi:(Y,\Delta)\to (Y',\Delta')$ be the  minimal model 
over $X$. 

(a) Then $\phi=\alpha\circ \upsilon$, where 

\begin{itemize}
\item $\upsilon$ is the successive contraction of those admissible twigs over $D$ of  $\Delta$, and of the  divisorial direct images of $\Delta$,   that can be contracted to smooth points, 
and 
\item $\alpha$ is the contraction of all maximal admissible twigs of $\upsilon_{*}\Delta$ over $D$.
\end{itemize}

(b) Let $\Upsilon=\Ex(\upsilon)$ be the exceptional locus of $\upsilon$.
Then  $\Upsilon$ is the sum of all maximal subtrees  $T$ of $\Delta-\mu^{-1}_{*}D$ such that $d(T)=1$ and $T$ meets $\Delta - T$ exactly once, in an irreducible component of $T$ of multiplicity one in the scheme theoretic inverse image $\upsilon^{-1}(\upsilon(T))$.

(c)  An irreducible component $E$ of $\Delta$ is not $\phi$-exceptional if and only if  $E\not\subseteq \Upsilon$ and $E$ satisfies one of the following:

\begin{enumerate} 
\item $E\subseteq \mu^{-1}_{*}D$, or 
\item $\beta_{\Delta-\Upsilon}(E)\geq 3$, or
\item $E$ is an irreducible component of a subchain $T$ of $\Delta-\Upsilon$ with $\beta_{{\Delta-\Upsilon}}(E)=2$, such that each tip of $T$ meets an irreducible component of $\Delta$ satisfying (1) or (2). 
\end{enumerate}

(d)  If $\mu$ is a composition of blowups at singular points of the respective reduced preimage of $D$, then $\upsilon=\mathrm{id}$ and $\Upsilon=0$.

\end{prop}

\begin{proof} Write $\phi=q\circ\sigma\circ p$, where $p\colon (Y,\Delta)\to (\bar{Y},\bar{B})$ is a contraction of some admissible twigs of $\Delta$ over $D$, $q$ is a birational morphism, and $\sigma$ is an extremal ray contraction obtained by contracting a prime divisor $\bar{A}$. Thus $\bar{A}^2<0$ and $\bar{A}\cdot (K_{\bar{Y}}+\bar{B})<0$. Write $a(p)=\sum_{E} a(E,\bar Y,\bar B) E,$ where the sum runs over all irreducible components $E$ of $\Ex(p)$, and $a(E,\bar Y,\bar B)$ is the log discrepancy. Put $B=p^{-1}_{*}\bar{B}$, $A=p^{-1}_{*}\bar{A}$. We have $0>\bar{A}^{2}=A\cdot p^{*}\bar{A}\geq A^{2}$ and $$0>\bar{A}\cdot (K_{\bar{Y}}+\bar{B})=A\cdot (K_{Y}+\Delta-a(p))=-2+2p_{a}(A)+\beta_{B}(A)+A\cdot (\Ex (p)-a(p)).$$ Note that $p_{a}(A)\geq 0$, and $\beta_{B}(A)\geq 1$ since the connected  component of $B$ containing $A$ also contains a component of $\mu^{-1}_{*}D$. Moreover, by Proposition \ref{propCQ} we have 
 $A\cdot (\Ex(p)-a(p))=\sum_{i=1}^{n}\left (1-{1}/{d(T_{i})}\right)\geq 0,$ where $T_1,\dots, T_n$ are the twigs of $\Delta$ meeting $A$ and contracted by $p$. Note that $T_i$ are also the connected components of $\Ex(p)$ meeting $A$.  It follows that $p_{a}(A)=0$ and $\beta_{B}(A)=1$, so $A$ is an admissible tip of $B$. Also $\sum_{i=1}^{n} (1-1/d(T_{i}))<1$. So we can assume that $d(T_{i})=1$ for $i>1$, that is,  $T_{2},\dots, T_{n}$ get contracted to smooth points. The image of $T_{1}+A$ after this contraction must then be an admissible twig of the image of $\Delta$.

It follows that $\alpha\circ\upsilon$ factors through $p\circ \sigma$. Applying this fact inductively, starting from $p=\mathrm{id}$, we see that $\alpha\circ\upsilon$ factors through $\phi$. Conversely, if $A$ is an irreducible component of $\Ex(\alpha\circ\upsilon)-\Ex(\phi)$, then reversing the above computation we see that $\bar{A}=\phi_{*}A$ satisfies $\bar{A}^{2}<0$, $\bar{A}\cdot (K_{Y'}+\Delta')<0$, which is impossible. Thus $\phi=\alpha\circ\upsilon$. This proves (a). Then (b) and (d) are easy consequences.

For (c), replace $(Y,\Delta)$ by $(\upsilon(Y),\upsilon_{*}\Delta)$, and assume $\upsilon=\mathrm{id}$. Let $R$ be the sum of all components of $\Delta$ satisfying (1)-(3). We need to prove that $\Delta-R=B$, where $B$ is the sum of all admissible twigs of $\Delta$ over $D$. Let $E$ be a component of $\Delta$. Clearly, if $E$ satisfies (1) or (2) then $E\subseteq \Delta-B$. Otherwise, let $T$ be a maximal subchain in the sense of inclusion of $\Delta$ containing $E$ whose irreducible components do not satisfy (1) or (2). Then either $T$ is as in (3), or one of the tips of $T$ is a tip of $\Delta$, and $T\subseteq \Delta-\mu^{-1}_{*}D$. Note that all irreducible components $C$ of $\Delta-\mu^{-1}_{*}D$ satisfy $C^2<0$, $C\simeq \mathbb{P}^1$. Since $\Delta-\mu^{-1}_{*}D$ is negative definite, $T$ is an admissible twig over $D$, so $E\subseteq B$.
\end{proof}

\begin{cor}\label{corSM}
Let $X$ be a smooth complex algebraic surface, $D$ a non-zero reduced divisor on $X$, and $\Sigma\neq\emptyset$ a Zariski closed subset of the support of $D$. Let $m>0$ be an integer. 

(a) Every dlt $m$-valuation of $(X,D,\Sigma)$ is an essential $m$-valuation. 

(b) Let $\mu:Y\to X$ be the minimal $m$-separating log resolution of $(X,D,\Sigma)$ and $\Delta=(\mu^*D)_{red}$. If a prime  divisor $E$ of $\mu$ corresponds to a dlt $m$-valuation of $(X,D,\Sigma)$, then $E$ does not get contracted on the minimal model of $(Y,\Delta)$ over $X$. The set of such prime divisors $E$ is given by the conditions that $\ord_E(D)$ divides $m$, $\mu(E)\subset \Sigma$, and
\begin{enumerate} 
\item $E\subset \mu_*^{-1}D$, or
\item $\beta_{\Delta}(E)\geq 3$, or
\item $E$ is an irreducible component of a subchain $T$ of $\Delta$ with $\beta_{{\Delta}}(E)=2$, such that each tip of $T$ meets an irreducible component of $\Delta$ satisfying (1) or (2). 
\end{enumerate}

\end{cor}
\begin{proof}
(a) Suppose that $\mu_1:Y_1\to X$ is another $m$-separating log resolution of $(X,D,\Sigma)$, and that $\mu_1=\mu\circ\pi$ for a morphism $\pi:Y_1\to Y$. Suppose $F$ is an irreducible component of $\Delta_1=(\mu_1^*D)_{red}$ such that $\mu_1(F)\subset\Sigma$, $\ord_F(D)$ divides $m$, and $F$ does not get contracted on the minimal model of $(Y_1,\Delta_1)$ over $X$. We can assume that $\pi$ is a composition of blowups of singular points of preimages of $D$, since blowing up smooth points of preimages of $D$ does not introduce exceptional divisors which survive on the relative minimal model by Proposition \ref{propTwig}. If $F$ is $\pi$-exceptional then there must exist two distinct irreducible components $E_1$ and $E_2$ of $\Delta$ such that $\pi(F)=E_1\cap E_2$. But this implies that $\ord_F(D)\ge N_1+N_2>m$ cannot divide $m$, which is a contradiction. Thus $F$ is not $\pi$-exceptional. Hence $F$ gives an essential $m$-valuation of $(X,D,\Sigma)$.

(b) If $E\subset \Delta$ corresponds to a dlt $m$-valuation of $(X,D,\Sigma)$, then $E=\pi(F)$ with $\pi$ and $F$ as above. By Proposition \ref{propTwig}, $F$ must satisfy one of the conditions (1)-(3) with $(\mu_1,Y_1,\Delta_1)$ replacing $(\mu, Y,\Delta)$. It is easy to see that this is equivalent to $E$ satisfying one of the conditions (1)-(3) for $(\mu, Y,\Delta)$.
Thus $E=\pi(F)$ does not get contracted on the minimal model of $(Y,\Delta)$ and the conditions (1)-(3) characterize such $E$. Alternatively, one also sees that $E$ does not get contracted on the minimal model of $(Y,\Delta)$ by applying the more general Lemma \ref{lemdlE}.
\end{proof}

\section{Curves on surfaces: contact valuations}\label{secCNP}	

We answer now the embedded Nash problem for irreducible formal plane curve germs.

	\subs{\bf  Notation.} \label{notation}
Let $(C,0) \subset (\bC^2,0)$ be an irreducible formal plane curve germ with $g\ge 1$ Puiseux pairs and multiplicity $\nu$. By finite determinacy, we can assume that $f$ is the germ of a polynomial. We take $x, y$ to be the coordinates on $\bC^2$.

	Let $m \geq 1$ be an integer.
	Let $\mu: Y \to \bC^2$ of $(C,0)$ be the minimal $m$-separating log resolution. It is obtained as a sequence of blow-ups of points, and we denote by $\mu_i: Y_i \to Y_{i-1}$ the $i$-th blow-up and by $E_i \subset Y_i$ its exceptional divisor. We set $Y_0:=\bC^2$. 	For every  divisor $E_i$ we have the corresponding divisorial valuation $v_i: \bC(x,y)^\times \to \Z $. Then $N_i=v_i(C)$ by definition.

	The resolution graph of $\mu$, denoted by $\Gamma_m$ as in Section \ref{secCur}, is as in Figure \ref{figr}. It is a refinement of the resolution graph $\Gamma$ of $C$. We denote as in the introduction the $g$ rupture components by $E_{R_1}, \ldots, E_{R_g}$, in the order of their appearance. For convenience, we set $C_0$ to be the $x$-axis and $E_{R_0}=E_0$ to be the $y$-axis.

		For a curve or arc $G$ in  $Y_{i'}$ we denote by $G^{(i)}$ the strict transform of $G$ in $Y_i$ with $i\ge i'\ge 0$. If it is clear from the context which ambient $Y_i$ is considered, we often write $G$ for $G^{(i)}$ in the case of exceptional divisors.

	For every $1\le j\le g$, let $(C_j,0)$ be the $j$-th approximate root of $(C,0)$, see \cite{PP}. It is a plane irreducible curve germ with $j$ Puiseux pairs whose resolution graph coincides with $\Gamma$ up to the $j$-th rupture component. In particular $C_j^{(R_j)}$ intersects $E_{R_j}$ in $Y_{R_j}$ transversely at the point of intersection of $E_{R_j}$ and $C^{(R_j)}$ for $0\le j\le g$.
	
	If $G, G'$ are two prime divisors in a smooth surface intersecting transversely at one point, the \emph{divisors  between $G$ and $G'$} are the (strict transforms of the) exceptional divisors that result from blowing up the point and, inductively, further intersection points of (the strict transforms of ) $G, G'$, and of previous exceptional divisors. The following is elementary and is part of the algorithm to compute resolution graphs, see \cite[p. 524]{BK}:

\begin{lemma} \label{blow-up-formula}
    Let $G, G'$ be two prime divisors intersecting transversely in one point $P$ in a smooth surface. Set $G = E_{(1,0)}, G' = E_{(0,1)}$, and then inductively set the divisor resulting from blowing up the intersection of $E_{(r,s)}$ and $E_{(r',s')}$ to be $E_{(r+r', s+s')}$. This establishes a bijection between pairs of coprime non-zero natural numbers $(r,s)$ and divisors between $G$ and $G'$.  Furthermore, if $x, y$ is a local system of coordinates at $P$ for which $G = \{y=0\}$ and $G' = \{x=0\}$, then the minimal composition of blow-ups that makes $E_{(r,s)}$ appear is given by
    $
    (x, y) = (\tilde{x}^s \tilde y^a, \tilde{x}^r \tilde y^b),
    $
    where $(a, b)\in\bN^2$  is the unique pair such that $ar - bs = (-1)^\ell, 0 \leq a \leq s, 0 \leq b \leq r$, where $\ell$ is the number of divisions in the Euclidean algorithm to compute  $\gcd(r, s)$. In these coordinates, $E_{(r,s)} = \{\tilde{x} = 0\}$ and $\{\tilde y=0\}$ is the strict transform of $E_{(b,a)}$.
\end{lemma}

	In terms of Figure \ref{figr}, in $\Gamma_m$, the divisors  between  $C_j^{(R_j)}$ and $E_{R_j}$ with $ 1\le j < g$ are those on the simple path joining $E_{R_j}$  with the bottom vertex of the $(j+1)$-st vertical group, excluding $E_{R_j}$. The divisors between  $C^{(R_g)}$ and $E_{R_g}$ are those to the right of $E_{R_g}$, excluding $E_{R_g}$ and $\tilde C$, and we call them the \emph{divisors after $E_{R_g}$}. The divisors between  $C_0$ and $E_{R_0}$, the original $x, y$ axes in our convention, are the first vertical group.

	We denote in this section the $m$-contact locus of $(\bC^2,C,0)$ by 
	$$\X_m(C) := \X_m(\bC^2,C,0) = \{\gamma \in \LL(\bC^2,0) \ |\ C \cdot \gamma = m\}.$$
The minimal $m$-separating resolution gives rise to the decomposition 
	$\X_m(C) = \sqcup_{i\in S_m} \X_{m,i}$ as in (\ref{eqD}).
	We  say that an arc $\gamma\in \X_m(C)$ \emph{lifts between $G$ and $G'$} if $G$ and $G'$ intersect transversally at a point on some $Y_i$ through which $\mu$ factors, and the lift of $\gamma$ to $Y$ intersects a divisor between $G$ and $G'$.

    \subs{\bf Decomposition of the contact loci.}
    \begin{theorem}\label{decomp} 
    For each $j=1,\ldots, g$, denote by $Z_j$ the subset of arcs in the contact locus $\X_m(C)$ which lift to one of the divisors in the $j$-th vertical group (including $E_{R_j}$ if it is an $m$-divisor) of $\Gamma_m$ as in Figure \ref{figr}. Denote by $S_m'$ the divisors in $S_m$ which are not in any vertical group of $\Gamma_m$. Then 
    \begin{equation} \label{decomposition}
    \X_m(C) = \bigsqcup_{j=1}^g Z_j \sqcup  \bigsqcup_{i\in S_m'} \X_{m,i}
    \end{equation}
    is a disjoint union decomposition into closed sets in the Zariski topology.
\end{theorem}

\begin{proof}
    It only remains to show that each of the sets in the decomposition, that is, $Z_j$ with $j \in \{1,\ldots,g\}$ and $\X_{m,i}$ with $i \in S'_m$, is closed in the Zariski topology. By \cite[Theorem A]{ELM} each of these subsets is a cylinder in $\LL(\C^2,0)$. Thus for $l \gg 0$, $Z^l_j \coloneqq \pi_l(Z_j)$ and $\X^l_{m,i} \coloneqq \pi_l(\X_{m,i})$ are constructible subsets in $\LL_l(\C^2,0)$. Moreover, $Z_j$ (resp. $\X_{m,i}$) is closed in $\LL(\C^2,0)$ if and only if $Z^l_j$ (resp. $\X^l_{m,i}$) is closed in $\LL_l(\C^2,0)$. Therefore it is enough to prove that each $Z^l_j$ and $\X^l_{m,i}$ is closed in the analytic topology of $\LL_l(\C^2,0)$.

    Arguing by contradiction, we suppose there are divisors $E_i$ and $E_{i'}$ such that $\X^l_{m,i}$ and $\X^l_{m,i'}$ are contained in different sets of the decomposition \eqref{decomposition} and $\X^l_{m,i} \cap \overline{\X^l_{m,i'}}\neq \varnothing$. Since $\X^l_{m,i}$ and $\X^l_{m,i'}$ are constructible, we may apply the Curve Selection Lemma  to obtain 
    a  $\bC$-analytic map germ $\alpha: (\C,0) \to \X^l_m$ such
    that $\alpha(0) \in \X^l_{m,i}$ and $\alpha(s) \in \X^l_{m,i'}$ for $s \neq 0$. Note that 
 since the $l$-jet space is finite-dimensional   
    the classical version of the Curve Selection Lemma is used here, not the more delicate version that appears often in the context of the infinite-dimensional arc spaces.  Using the zero section of the truncation map $\LL(\C^2,0) \to \LL_l(\C^2,0)$, we may regard $\alpha$ as a $\bC$-analytic map germ to $\LL(\C^2,0)$ such that for every $s$, the entries of the arc $\alpha(s)$ are all polynomials in $t$.

    Let $f$ be an equation for $C$. Note that $f(\alpha(s))$ is a $\bC$-analytic map germ of polynomials of constant order $m$. Therefore we can find a  $\bC$-analytic  map germ $u: (\C,0) \to \LL(\C,0)$ such that for every $s$, $u(s)$ is convergent and $f(\alpha(s)) = u(s)^m$. To see this, we can write $u(s)$ explicitly by picking a holomorphic $m$-th root of the coefficient of $t^m$ in $f(\alpha(s))$ and using the general formulas relating the coefficients of both sides of
    $
    (\sum_{k\ge 0}b_kt^k )^m= \sum_{k\ge 0}c_kt^k$ with $c_0\neq 0,
    $
 namely, $c_0=b_0^m$, $c_k=\frac{1}{b_0}\sum_{i=1}^k(im-k+i)b_ic_{k-i}$ for $k\ge 1$.
 This shows that $u$ is indeed analytic. 
    Also note that $\ord_t(u(s)) = 1$, so we can make the change of variables $(\tilde{s}, \tilde{t}) = (s, u(s)(t))$, and then $f(\alpha(\tilde{s}, t(\tilde{s},\tilde{t}))) = \tilde{t}^m$. In summary, after a reparametrization we can  suppose that for every $s$, $\alpha(s)$ is a convergent arc with $f(\alpha(s)) = t^m$ and such that $\alpha(0) \in \X^l_{m,i}$ and $\alpha(s) \in \X^l_{m,i'}$ for $s \neq 0$.

    We can assume that $\al(s)$ is defined for $s\in[0,1]$. For any non-constant arc $\gamma \in \LL(\C^2,0)$, denote by $\tilde{\gamma}$ the lift of $\gamma$ to the resolution. Let $p_0 \in E_i^{\circ}$ and $p_1 \in E_{i'}^{\circ}$ be the centers of $\tilde{\al}(0)$ and $\tilde{\al}(1)$ respectively, where $E_{k}^{\circ}\coloneqq E_{k}\setminus \bigcup_{k'\neq k} E_{k'}$. There exist coordinate charts $(U_0,\psi_0)=(z_1^{(0)},z_2^{(0)})$ and $(U_1,\psi_1)=(z_1^{(1)},z_2^{(1)})$ at $p_0$ and $p_1$, respectively, such that $\psi_0(p_0) = (0,0)$ and $f \circ \mu|_{U_0}(z_1^{(0)}, z_2^{(0)}) = (z_1^{(0)})^{N_i}$; and analogously for $(U_1,\psi_1)$. After possibly reducing the radius of convergence of the arcs in the family $\alpha$, we may assume that $U_0$ and $U_1$ contain the images of $\tilde{\al}(0)$ and $\tilde{\al}(1)$, respectively.

    We write $\tilde{\al}(0)$ in the local coordinates of $(U_0,\psi_0)$ as $\tilde{\al}(0)(t)=(a_1(t), a_2(t))$. Note that $a_1(t)^{N_i} = f(\alpha(0)(t)) = t^m$, and therefore $a_1(t) = t^{m/N_i}$, after possibly multiplying the coordinate $z_1^{(0)}$ by a root of unity, which does not affect the properties of $\psi_0$. Hence $(a_1(t), \lambda a_2(t))$ with $\lambda \in [0,1]$ defines a continuous 
    path in $\X_m$ joining $\tilde{\alpha}(0)$ and the arc which in the coordinates of $(U_0,\psi_0)$ is given by $(t^{m/N_i}, 0)$. After a similar construction for $\alpha(1)$ and pushing the paths to $\C^2$ via $\mu$, we can concatenate and obtain a path $\beta: [0,1] \to \X_m$ such that for every $s$, $\beta(s)$ is convergent and $f(\beta(s)) = t^m$, and in the local coordinates of $(U_0,\psi_0)$ and $(U_1,\psi_1)$ we have $\tilde{\beta}(0) = (t^{m/N_i},0)$ and $\tilde{\beta}(1)=(t^{m/N_{i'}},0)$ respectively. From now on we write $\beta(s,t)$ for $\beta(s)(t)$.
\smallskip 

    Let $0 < \delta \ll \varepsilon \ll 1$ be Milnor radii for $f$. Denote  by $B_\varepsilon$ the open ball in $\C^2$ centered at the origin of radius $\varepsilon$ and by $D_\delta$ the open disk in $\C$ centered at the origin of radius $\delta$. We pull-back the Milnor fibration to $Y$ via $\mu$, that is, we put $Y'\coloneqq \mu^{-1}(B_\varepsilon \cap f^{-1}(D_\delta))$ and consider the map $f \circ \mu\colon  Y'\to D_\delta$. Its restriction over $S^1_\delta$ is a locally trivial fibration isomorphic to the Milnor fibration (in the tube). 
    
    We now recall a construction of a topological model for this fibration, given by A'Campo in \cite[\S 2]{AC}. Let $\tau\colon Y_{\log}\to Y'$ be a real oriented blowup along $(f\circ \mu)^{-1}(0)$. Let $\eta\colon A\to Y_{\log}$ be a continuous map which is an identity over $\bigsqcup_{k}\tau^{-1}(E_{k}^{\circ})$, and for each pair $E_{k},E_{k'}$ it replaces $\tau^{-1}(E_{k}\cap E_{k'})$ by $\tau^{-1}(E_{k}\cap E_{k'})\times [0,1]$, see \cite[p.\ 239]{AC}. The resulting topological space $A$ is a manifold with boundary and corners, equipped with a continuous map $\pi_{A}\coloneqq\tau\circ \eta\colon A\to Y$ and a locally trivial fibration $f_{A}\colon \partial A\to S^{1}$ which is topologically equivalent to the Milnor fibration in the tube, i.e.\ admits a homeomorphism $h\colon \mu^{-1}(B_\varepsilon \cap f^{-1}(S^1_\delta))\to \partial A$ such that $f_{A}\circ h=\frac{1}{\delta}\cdot f\circ \mu$. Let $(\phi_\theta\colon \F_1 \to \F_{e^{2\pi \imath \theta}})_{\theta \in \R}$, where $\F_{z}\coloneqq f_{A}^{-1}(z)$, be the monodromy trivialization of $f_{A}$. We now list some useful properties of the decomposition $\partial A=\bigsqcup_{k} \pi_{A}^{-1}(E_{k}^{\circ})\sqcup\bigsqcup_{k,k'} \pi_{A}^{-1}(E_{k}\cap E_{k'})$, see Figure \ref{figr-milnor}.
    
    For each $k$, the preimage $\pi_{A}^{-1}(E_{k}^{\circ})$ is a product $E_{k}^{\circ}\times S^{1}$, and under this identification we have $(\pi_{A},f_{A})(x,z)=(x,z^{N_{k}})$, so the monodromy $\phi_{\theta}$ rotates the second coordinate by an angle $\frac{2\pi \theta}{N_{k}}$. We refer to $\pi_{A}^{-1}(E_{k}^{\circ})\cap \F_{1}$ as the \emph{piece of $\F_{1}$ over $E_{k}$}; in this piece we have $\phi_{N_{k}}=\mathrm{id}$. Note that if the vertex of $\Gamma_{m}$ corresponding to $E_{k}$ has valence $1$ (respectively, $2$) then each connected component of the piece of $\F_{1}$ over $E_{k}$ is a disk (respectively, a cylinder).
    
    Recall that we have distinguished a point $p_{0}\in E_{i}^{\circ}$ and a chart $\psi_{0}\colon U_0\to \C^2$ around $p_0$, with coordinates $(z_{1}^{(0)},z_{2}^{(0)})$, such that $f\circ \mu\circ \psi_{0}^{-1}=(z_{1}^{(0)})^{N_i}$. 
    Writing $z_{1}^{(0)}$ in polar coordinates yields a chart $\psi_{0}'\colon \pi_{A}^{-1}(U_{0})\to [0,\infty)\times S^1\times \C$ such that $\pi_{A}((\psi_{0}')^{-1}(r_{1},e^{2\pi \imath\, \theta_{1}},z_{2}^{(0)}))=\psi_{0}^{-1}(r_{1}\cdot e^{2\pi \imath\, \theta_{1}},z_{2}^{(0)})$. Since the homeomorphism $h$ can be defined by locally lifting the radial vector field on the real oriented blowup of the base $D_{\delta}$,  shrinking $U_{0}$ if needed we can assume that $h|_{U_0\cap f^{-1}(S^1_{\delta})}$ is a translation in the $r_{1}$-direction, i.e.\ $h(\psi_{0}^{-1}(\delta^{1/N_{i}}\cdot e^{2\pi \imath \cdot \theta_{1}},z_{2}^{(0)}))=(\psi_{0}')^{-1}(0, e^{2\pi\imath\cdot \theta_{1}},z_{2}^{(0)})$. As a consequence, 
    \begin{equation}
    	\label{compatibility}
    	\phi_\theta(h(\psi_0^{-1}(z_1^{(0)},z_2^{(0)}))) = h(\psi_0^{-1}(z_1^{(0)}e^{2 \pi \imath\, \theta/N_i},z_2^{(0)})),
    \end{equation}
    and an analogous formula holds for the chart $U_{1}$ around $p_{1}\in E_{i'}^{\circ}$.
    
    \begin{figure}[ht] 
    	\centering
    	\resizebox{.7\linewidth}{!}{%
    \fontsize{20pt}{20pt}\selectfont
    \def\svgheight{0.5cm}
    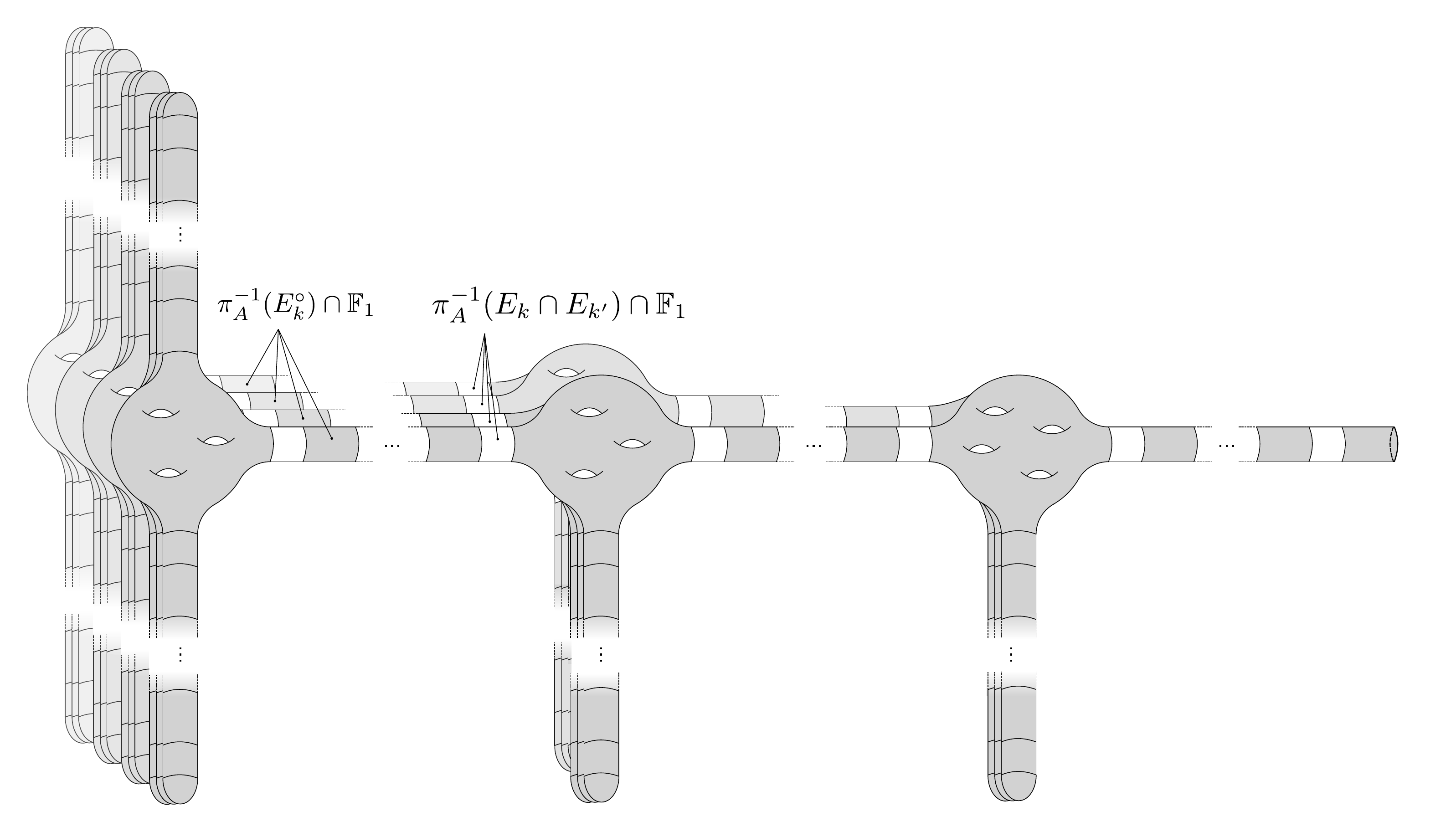
}
    	\caption{A'Campo decomposition of the Milnor fiber using the minimal $m$-separating log resolution, see \cite[\S 2]{AC}. Note the resemblance with Figure \ref{figr}.}
    	\label{figr-milnor}
    \end{figure}

    Consider now an intersection point $E_{k}\cap E_{k}'$. Each connected component of $\pi_{A}^{-1}(E_{k}\cap E_{k}')$ is a product $(S^1\times S^{1})\times [0,1]$. With an additive notation in each $S^1=\R/2\pi\Z$, the monodromy $\phi_{\theta}$ reads as $(\theta_{1},\theta_{2},\lambda)\mapsto (\theta_{1}+\lambda \cdot \frac{2\pi\theta}{N_{k}}, \theta_{2}+(1-\lambda)\cdot \frac{2\pi\theta}{N_{k'}}, \lambda)$, see  \cite[p.\ 240]{AC}. 
    We call $\F_{1}\cap \pi_{A}^{-1}(E_{k}\cap E_{k'})$ the \emph{piece of $\F_{1}$ between $E_{k}$ and $E_{k'}$}. In the above coordinates it is given by $\{N_{k}\theta_{1}+N_{k'}\theta_{2}=0\}$, so it is a disjoint union of cylinders $S^{1}\times [0,1]$. Putting $N=\operatorname{lcm}(N_k,N_{k'})$ we see that this piece is invariant under $\phi_{N}$, and $\phi_{N}$ restricts to a Dehn twist on each cylinder.

    Consider a homeomorphism of $\F_{1}$ defined as the above Dehn twist in $\F_{1}\cap \pi_{A}^{-1}(E_{k}\cap E_{k'})$ and as the identity elsewhere. It is isotopic to $\mathrm{id}_{\F_{1}}$ if and only if the image of $\pi_{1}(\F_{1}\cap \pi_{A}^{-1}(E_{k}\cap E_{k'}))$ in $\pi_{1}(\F_{1})$ is trivial. This happens if $E_k$ and $E_{k'}$ are in the same vertical group of $\Gamma_m$, see Figure \ref{figr}. Indeed, in this case each connected component of $\pi_{A}^{-1}(E_{k}\cap E_{k'})$ lies in a disk obtained as follows, see Figure \ref{figr-milnor}: to a disk over a divisor corresponding to a vertex of valence one we attach, along the boundary, a sequence of cylinders lying either between divisors, or over ones corresponding to vertices of valence $2$.  
    Conversely, if $E_k$ and $E_{k'}$ are not both contained in the same vertical group of $\Gamma_m$, then the image of $\pi_{1}(\F_{1}\cap \pi_{A}^{-1}(E_{k}\cap E_{k'}))$ in $\pi_{1}(\F_{1})$ is non-trivial, so the above Dehn twist 
    is not isotopic to the identity.
\smallskip
   

   We now return to studying our path $\beta\colon [0,1] \to \X_m$. We can assume $\beta(s)$ is defined for $|t|\le \delta^{1/m}$ for every $s$. Consider the continuous map
    \[
    \sigma\colon [0,1] \times [0,1] \to \F_1, \quad (s, \theta) \mapsto \phi^{-1}_{m\theta}(h(\tilde{\beta}(s, \delta^{1/m}\cdot e^{2\pi\imath\, \theta}))),
    \]
    which is well defined because $f(\beta(s,\delta^{1/m}e^{2\pi\imath\, \theta})) = \delta e^{2\pi\imath m\, \theta}$. This map has the property that $\sigma(0,\theta)$ and $\sigma(1,\theta)$ are independent of $\theta$. Indeed, using formula (\ref{compatibility}) we get
    \[
    \sigma(0,\theta) = \phi^{-1}_{m\theta}(h(\tilde{\beta}(0, \delta^{1/m}e^{2\pi\imath\, \theta}))) = \phi^{-1}_{m\theta}(h(\psi_0^{-1}(\delta^{1/N_i}e^{2\pi\imath m\, \theta/N_i},0))) = h(\psi_0^{-1}(\delta^{1/N_i},0)),
    \]
    and analogously for $\sigma(1,\theta)$. Therefore $\sigma$ is a homotopy relative to the endpoints between the paths $\sigma_0(s) \coloneqq \sigma(s,0)$ and $\sigma_1(s) \coloneqq \sigma(s,1)$. Moreover, the starting point lies on the piece of $\F_1$ over $E_i$ and the endpoint lies on the piece of $\F_1$ over $E_{i'}$. Note that by definition we have $\phi_m \circ \sigma_1 = \sigma_0$.

    As explained before, if $N$ is a common multiple of all $N_k$ with $k\ge 0$, then $\phi_{Nm}$ is a composition of Dehn twists. Since $\phi_{Nm} = \phi_m^N$, it follows  that $\phi_{Nm} \circ \sigma_1$ is homotopic to $\sigma_1$ relative to the endpoints. Recall that the Dehn twists that are isotopic to the identity are the ones occurring in the connected components of the pieces of $\F_1$ between two divisors  of a vertical group of $\Gamma_m$. If we denote by $\widetilde{\phi}_{Nm}$ the composition of those Dehn twists appearing in $\phi_{Nm}$ which are not isotopic to the identity, we have that $\widetilde{\phi}_{Nm}$ and $\phi_{Nm}$ are isotopic. Although it is possible that $\sigma_1(0)$ and $\sigma_1(1)$ are not fixed by this isotopy, we claim that the path $\widetilde{\phi}_{Nm} \circ \sigma_1$ is homotopic to $\sigma_1$ relative to the endpoints. Indeed, note that the isotopy is equal to the identity on the pieces of $\F_1$ over the divisors in the horizontal group of $\Gamma_m$, including the rupture components. Therefore, if  $\sigma_1(0)$ and $\sigma_1(1)$ are contained in pieces of $\F_1$ over divisors in the horizontal group, then it is immediate that $\widetilde{\phi}_{Nm} \circ \sigma_1$ is homotopic to $\sigma_1$ relative to the endpoints. If either $\sigma_1(0)$ or $\sigma_1(1)$ lies in a piece of $\F_1$ over a divisor in a vertical group of $\Gamma_m$ distinct from the rupture divisors, then the homotopy between $\widetilde{\phi}_{Nm} \circ \sigma_1$ and $\phi_{Nm} \circ \sigma_1$ induces a loop contained in a disk in $\F_1$ based at $\sigma_1(0)$ and a loop contained in a disk in $\F_1$ based at $\sigma_1(1)$. By Lemma \ref{homotopies} below, also in this case $\widetilde{\phi}_{Nm} \circ \sigma_1$ is homotopic to $\sigma_1$ relative to the endpoints.
    
    Moreover, the endpoints $\sigma_{1}(0)$ and $\sigma_{1}(1)$  lie in the pieces of $\F_{1}$ over $E_{i}$ and $E_{i'}$,  which are different connected components of the fixed point set of $\widetilde{\phi}_{Nm}$. Indeed, $\operatorname{Fix}{\widetilde{\phi}_{Nm}}$ consists of all pieces of $\F_{1}$ lying over divisors; all pieces lying between adjacent divisors in the same vertical group, and possibly some circles in the interior of pieces between other divisors, where $\widetilde{\phi}_{Nm}$ is an iterate of a Dehn twists. The claim follows since by assumption  $E_{i}$ and $E_{i'}$ do not lie in the same vertical group.
    
    Therefore the properties of $\sigma_1$ contradict Lemma \ref{seidel} below and the proof is concluded.
\end{proof}

In the above proof, we have used the following two elementary lemmas. 
The second one is a consequence of \cite[Lemma 3 (ii)]{Se}.

\begin{lemma} \label{homotopies}
    Let $T$ be a topological space and let $\eta_0$ and $\eta_1$ be paths in $T$ with the same endpoints. Suppose there exists a homotopy $H: [0,1] \times [0,1] \to T$ between $\eta_0$ and $\eta_1$, that is, $H(\cdot,0) = \eta_0$ and $H(\cdot,1) = \eta_1$, such that $H(0,\cdot)$ and $H(1,\cdot)$ are contractible loops. Then $\eta_0$ and $\eta_1$ are homotopic relative to the endpoints.
\end{lemma}

\begin{lemma} \label{seidel}
    Let $\Sigma$ be a surface and $T = T_1 \circ \dots \circ T_r$ be a composition of Dehn twists such that
    \begin{itemize}
        \item no $T_i$ is isotopic to the identity, and
        \item if the supports of $T_i$ and $T_j$ bound an annulus, then $T_i$ and $T_j$ have the same orientation.
    \end{itemize}
     If $\eta: [0,1] \to \Sigma$ is a path whose endpoints are fixed by $T$ and such that $T \circ \eta$ is homotopic to $\eta$ relative to the endpoints, then $\eta(0)$ and $\eta(1)$ are in the same connected component of the set of fixed points of $T$.
\end{lemma}

	\subs{\bf  Irreducibility.} We finish now the proof of Theorem \ref{thmNPCur} by showing that the sets $\sX_{m,i}$, $i \in S_m'$, and $Z_j$, $j \in \{1,\dots,g\}$, are all irreducible. The former are already known to be irreducible, so we focus on the sets $Z_j$.

	We assume first that $C$ has only one Puiseux pair, that is, up to higher order terms it has the form $y^p =x^q$ with $\gcd(p, q) = 1$. The only  rupture component is denoted $E_R$. It will be useful  to also consider  smooth curves which are tangent to the $y$-axis, that is,  $q>p = 1$. In this case by minimal resolution we mean the composition of the $q$ blow-ups needed to remove the tangency, and the rupture component is the last of these exceptional divisors. 
	
	\begin{lemma} \label{ishii-lemma} Assume $g=1$.
		Let $E_i,E_{i'}$ be two divisors  in the first  vertical group in Figure \ref{figr}, and suppose $N_i$ and  $N_{i'}$ divide $m$. If 
		${m} v_i(G)/{N_i} \leq {m} v_{i'}(G)/{N_{i'}}$
		for all $G \in \bC[x,y]$, then $\overline{\sX_{m,i}} \supset {\sX_{m,{i'}}}$.
	\end{lemma}

    \begin{proof}
        Since $g=1$, all exceptional divisors give rise to toric  valuations. The claim is then proven by Ishii \cite[Lemma 3.11]{Is}. Indeed, in the notation of \cite{Is}, $C_X(mv_i/N_i)$ is the closure of the fat arcs whose associated valuation equals $mv_i/N_i$. In terms of $\C$-points, those are the arcs that lift to $E_i$ and intersect it with multiplicity $m/N_i$, that is, $C_X(mv_i/N_i) = \overline{\sX_{m,i}}$.
    \end{proof}
	
	\begin{proposition}
		\label{1pp} Assume $g=1$.
		If $Z_1$ is nonempty, there exists precisely one component $E_{i_0}$, such that $Z_1 = \overline{\sX_{m,i_0}}$. If $R\in S_m$, then $i_0 = R$. Otherwise there are $m$-valuations only on one side of the rupture divisor, and $E_{i_0}$ is the one among them which lies closest to $E_R$.
	\end{proposition}
	
	\begin{proof}
		Using the notation $E_{(r,s)}$ that we introduced above, note that
		$
		v_{(r,s)}(G) = \min\{r \cdot \beta + s \cdot \alpha \mid x^\alpha y^\beta \text{ is a monomial appearing in } G\}.
		$
		Since all valuations we are interested in are of this form, we only need to check their values at the two functions $x$ and $y$, since
		$
		v_i(x) \leq v_{i'}(x) \text{ and } v_i(y) \leq v_{i'}(y)$ implies $v_i(G) \leq v_{i'}(G)$ for all $G \in \bC[x,y].
		$
		Applying this to our  curve $C$, we see that
		$
		N_{(r,s)} =  \min\{pr, qs\}.
		$
		By Lemma \ref{ishii-lemma}, to find a divisor that dominates all others we 
		need to simultaneously minimize the functions
		\begin{align*}
			F_x={m} v_{(r,s)}(x)/{N_{(r,s)}} &= {ms}/{\min\{pr, qs\}} = {m(1-t)}/{\min\{pt, q(1-t)\}}, \\
			F_y={m} v_{(r,s)}(y)/{N_{(r,s)}} &= {mr}/{\min\{pr, qs\}} = {mt}/{\min\{pt, q(1-t)\}},
		\end{align*}
		where $t = r/(r+s) \in [0,1] \cap \Q$ represents all possible pairs of coprime numbers. The graphs  are as in Figure \ref{fig:function-plots}. Both functions reach a minimum at $t = q/(p+q)$, which corresponds  to $E_R$. However, this might not be an $m$-valuation. 
	
		\begin{figure}[ht]
			\centering
			\subfigure{
				\begin{tikzpicture}[scale=0.3]
					\begin{axis}[
						xmin = 0, xmax = 1,
						ymin = 0, ymax = 20,
						axis lines=middle,
						axis line style={->},
						x label style={at={(axis description cs:0.5,-0.02)},anchor=north},
						y label style={at={(axis description cs:-0.1,0.4)},anchor=south},
						xlabel={\scalebox{2}{$t = {r}/({r+s})$}},
						ylabel={\scalebox{2}{$F_x$}}]
						\addplot[
						domain = 0.05:1,
						samples = 100,
						smooth,
						thick,
						blue,
						] {6*(1-x)/(min(2*x, 3*(1-x)))};
					\end{axis}
				\end{tikzpicture}
			}
			\quad\quad
			\subfigure{
				\begin{tikzpicture}[scale=0.3]
					\begin{axis}[
						xmin = 0, xmax = 1,
						ymin = 0, ymax = 20,
						axis lines=middle,
						axis line style={->},
						x label style={at={(axis description cs:0.5,-0.02)},anchor=north},
						y label style={at={(axis description cs:-0.1,0.4)},anchor=south},
						xlabel={\scalebox{2}{$t = {r}/({r+s})$}},
						ylabel={\scalebox{2}{$F_y$}}]
						\addplot[
						domain = 0.01:0.95,
						samples = 100,
						smooth,
						thick,
						blue,
						] {6*x/(min(2*x, 3*(1-x)))};
					\end{axis}
				\end{tikzpicture}
			}
			\caption{}\label{fig:function-plots}
		\end{figure}
	For $v_{(r,s)}$ to be an $m$-valuation,$N_{(r,s)}$ must divide $m$. If $p \nmid m$ and $q \nmid m$, then  $Z_1$ is empty. If $p \mid m$ and $q \mid m$, since $N_{(q,p)}  = pq$,  $E_R$ gives an $m$-valuation and   $Z_1 = \overline{\sX_{m,R}}$. If $p \mid m$, $q \nmid m$, and $N_{(r,s)} \mid m$, then $pr < qs$. Equivalently ${r}/({r+s}) < {q}/({p+q})$, that is, all $m$-valuations appear to the left of the common minimum in the graphs of Figure \ref{fig:function-plots}. In particular,  one of them  is smaller than all others, so its corresponding component dominates the rest, and  $Z_1$ is the closure of the corresponding $\sX_{m,i_0}$. Moreover,  the  order of the divisors in the resolution graph is the same as the order induced by their labels $(r,s)$ when we write them as ${r}/({r+s})$. Thus the valuation that dominates all others, the one closest to the common minimum in the graphs of Figure \ref{fig:function-plots}, is also the closest $m$-valuation to $E_R$. If $p \nmid m$ and $q \mid m$ the situation is analogous. 
	\end{proof}

	This proves Theorem \ref{thmNPCur} in the case of one Puiseux pair. We return now to the  general case where $C$ has $g$ Puiseux pairs. We denote by $\nu$ the multiplicity of $C$ and by $k_1, \ldots, k_g$ its characteristic exponents. Let $r_1 = \nu, r_{j+1} = \gcd(k_j, r_j)$ for $j=1, \ldots, g$; and  $k_0=0$, $\kappa_j = k_j - k_{j-1}$ for $j=1, \ldots, g$, $\kappa_{g+1} = 1$.

    	\begin{lemma} \label{vC-computation}
		Let $1 \leq j < g$. Consider the parametrization at the $j$-th rupture component $E_{R_j} = \{x_j = 0\}$ of the strict transform $C^{(R_j)}$: $x_j(\tau) = \tau^{r_{j+1}} u_x(\tau)$, $y_j(\tau) = \tau^{\kappa_{j+1}} u_y(\tau)$, where $u_x, u_y$ are units. Let $c_0 = \kappa_{j+1}, c_1 = r_{j+1}$, and consider the steps of the Euclidean algorithm for $(c_0,c_1)$:
$c_{i-1} = \eta_i c_i + c_{i+1}$ with $0 < c_{i+1} < c_i$ for $i=1,\ldots, l-1$, and $c_{\ell-1} = \eta_{\ell} c_{\ell}$.
				As in \cite[p.522]{BK}, label by $E_{a,b}$ with  $a = 1, \ldots, \ell$, $b = 1, \ldots, \eta_a$, the exceptional divisors in the minimal log resolution of $C$ that lie  between  $C_j^{(R_j)}$ and $E_{R_j}$ by their order of appearance. Let  $(r,s)$ be the pair of coprime numbers such that $E_{a,b} = E_{(r,s)}$ as in Lemma \ref{blow-up-formula}. 
Then,
		\[
		N_{a,b}= \begin{cases}
			s(\kappa_{j+1} + N_{R_j}) & \text{if $a$ is even,} \\ 
			s(\kappa_{j+1} + N_{R_j}) + (\eta_1 c_1 + \eta_3 c_3 + \cdots + \eta_{a-2} c_{a-2} + b c_{a} - \kappa_{j+1}) & \text{if $a$ is odd.}
		\end{cases}
		\]									
	\end{lemma}

	
	\begin{proof}
		We do the case where $\kappa_{j+1} > r_{j+1}$, the other case being analogous. This is a straightforward computation, given the formulas to compute the multiplicities:
		\begin{align*}
			v_{1,1}(C) &= v_{R_j}(C) + \mult(C^{(R_j)}) = c_1 + M, \\
			v_{1,b}(C) &= v_{1,b-1}(C) + \mult(C^{(1,b-1)}) = b c_1 + M, \text{ for } b>1, \\
			v_{2,1}(C) &= v_{1,\eta_1}(C) + \mult(C^{(1, \eta_1)}) = \eta_1 d_1 + d_2 + M = \kappa_{i+1} + M, \\
			v_{2,b}(C) &= v_{1,\eta_1}(C) + v_{2,b-1}(C) + \mult(C^{(2,b-1)}) = b(\kappa_{i+1} + M), \text{ for } b>1, \\ 
			v_{a,1}(C) &= v_{a-2, \eta_{a-2}}(C) + v_{a-1,\eta_{a-1}}(C) + \mult(C^{(a-1, \eta_{a-1})}), \text{ for } a>2, \\
			v_{a,b}(C) &= v_{a,b-1}(C) + v_{a-1,\eta_{a-1}}(C) + \mult(C^{(a,b-1)}), \text{ for } a>2, b>1.
		\end{align*}
		The factor $s$ in the formulas appears because $s = v_{(r,s)}(x_j) =  v_{a,b}(E_{R_j})$ and this last multiplicity is computed with the same formulas as above, except that we do not add the multiplicity of $C$ at each step. As a result, $s$ is the coefficient of $M$ in $v_{a,b}(C)$, giving the desired result.
	\end{proof}
	
	\begin{lemma} \label{same-inters}
		If $\gamma \in \sX_m(C)$ lifts to a component $E_{(r,s)}$ between $C_j^{(R_j)}$ and $E_{R_j}$ with $0 < j < g$, then
		$
		\gamma^{(R_j)} \cdot E_{R_{j}} \geq {m}/(\kappa_{j+1} + N_{R_j}),
		$
		with equality if and only if $E_{(r,s)}$ is in the $(j+1)$-th vertical group of $\Gamma_m$.
	\end{lemma}
	
	\begin{proof}
		We have
		$
		\gamma^{(R_{j})} \cdot E_{R_{j}} = v_{(r,s)}(E_{R_{j}})(\gamma^{(r,s)} \cdot E_{(r,s)}) = s {m}/N_{(r,s)}.
		$
		So we need  that
		$
		N_{(r,s)} \leq s(\kappa_{j+1} + N_{R_j}),
		$
		with equality if and only if $E_{(r,s)}$ is in the vertical group. This holds by Lemma \ref{vC-computation} for  divisors  in the minimal resolution. In general, we argue by induction. Suppose $E_{(r'',s'')}$ has been obtained after blowing up $E_{(r,s)}\cap E_{(r',s')}$ and that $E_{(r,s)}$ and $E_{(r',s')}$ satisfy the claim. Then $(r'',s'') = (r+r', s+s')$ and its order of vanishing along $C$ is $N_{(r'',s'')} = N_{(r,s)} + N_{(r',s')}$. Hence
        \[
        N_{(r'',s'')} = N_{(r,s)} + N_{(r',s')} \leq s(\kappa_{j+1}+N_{R_j}) + s'(\kappa_{j+1}+N_{R_j}) = (s+s')(\kappa_{j+1}+N_{R_j}).
        \]
        Equality holds if and only if $N_{(r,s)} = s(\kappa_{j+1}+N_{R_j})$ and $N_{(r',s')} = s'(\kappa_{j+1}+N_{R_j})$; i.e. if both $E_{(r,s)}$ and $E_{(r',s')}$ are in the vertical group; i.e. if $E_{(r'',s'')}$ is in the vertical group.
	\end{proof}
	
	 We say that  $E_i\subset Y$  appears {\it no later than $E_{R_j}$} if it is a divisor between $C_k^{(R_k)}$ and $E_{R_k}$ for some $0\le k<j$; in Figure \ref{figr} these are all the vertices to the left of, under, or equal to $E_{R_j}$.
	
	\begin{lemma} \label{valuation-char-roots}
		Let  $1\le j\le g$. For every $E_i$ no later than $E_{R_j}$, $N_i=r_{j+1}v_i(C_j)$.  In particular, the minimal $(m/r_{j+1})$-separating  resolution graph of $C_j$ coincides with the part no later than $E_{R_j}$ of the minimal $m$-separating  resolution graph of $C$.
	\end{lemma}
	
	\begin{proof} 		The second claim  follows from the first one.
    It suffices to prove the first claim when $E_i$ is appears already in the minimal resolution.	Pick an arc $\tilde\gamma$ in $Y_{R_j}$ intersecting $E_i$ transversally but not intersecting any other irreducible component of the total transform of $C$. Let $\gamma$ be the image of $\tilde \gamma$ in $\bC^2$. Note that $C\cdot\gamma=(C^{(1)}+(\mult C)\cdot E_1)\cdot \gamma^{(1)}=C^{(1)}\cdot\gamma^{(1)}+(\mult C)\cdot (\mult \gamma)$. Inducting, we get
		$N_i = C \cdot \gamma = \sum_{k=1}^{i} (\mult C^{(k-1)})\cdot (\mult \gamma^{(k-1)}),$
		and similarly,
		$v_i(C_j) = C_j \cdot \gamma = \sum_{k=1}^{i} (\mult C_j^{(k-1)})\cdot (\mult \gamma^{(k-1)}).$
		Thus it is enough to show that
		$\mult C^{(k-1)} = r_{j+1} \cdot\mult C^{(k-1)}_j$
		for every $1\le k\le R_j$. Equivalently, that
		$C^{(k)} \cdot E_k = r_{j+1} C^{(k)}_j \cdot E_k.$
		Let $\mu_{R_j,k}=\mu_{k+1} \circ \dots \circ \mu_{R_j}$, where $\mu_i$ is the sequence of  blow-ups forming the minimal resolution. Then $C^{(k)} \cdot E_k=C^{(R_j)} \cdot \mu_{R_j,k}^*E_k =v_{R_j}(E_k)(C^{(R_j)} \cdot E_{R_j})$ and one has $v_{R_j}(E_k)=C^{(k)}_j \cdot E_k$. By  \cite[Theorem 5.1]{PP}  an integer $n$ is in the semigroup of $C$ if and only if $r_{j+1}n$ is in the semigroup of $C_j$. The  multiplicity of a branch is the least positive number of its semigroup, so $\mult C = r_{j+1}\cdot \mult C_j$. From above for $k=1$, we had $  \mult C = (C^{(R_j)} \cdot E_{R_j}) \cdot(\mult C_j)$,		hence $C^{(R_j)} \cdot E_{R_j}=r_{j+1}$.		\end{proof}
	
	\begin{proposition}
		If $Z_1$ is nonempty, there exists precisely one component $E_{i_0}$ no later than $E_{R_1}$ such that $Z_1 = \overline{\sX_{m,i_0}}$. If $N_{R_1}$ divides $m$, then $i_0 = R_1$. Otherwise there are $m$-valuations only on one side of the rupture divisor $E_{R_1}$, and $E_{i_0}$ is the one among them which lies closest to $E_{R_1}$.
	\end{proposition}
	
	\begin{proof} If $E_i$ in the first vertical group,  then $m/N_i=(m/r_2)/v_i(C_1)$
by Lemma \ref{valuation-char-roots}. Thus  $C_1$ is a curve with one Puiseux pair whose minimal $(m/r_2)$-separating log resolution graph coincides with the minimal $m$-separating log resolution graph of $C$ up to the first rupture component $E_{R_1}$, and
  $\sX_{m,i}(C)=\sX_{{m}/{r_2},i}(C_1)$. So the claim follows from Proposition \ref{1pp} for $C_1$.
	\end{proof}

	 \begin{lemma} \label{equimultiplicity}
	    Let $\sigma: B \to \bC^2$ be the blow-up  at the origin. Fix  $n\in\bN$ and let $M$ be the set of arcs in $\bC^2$ of multiplicity $n$ at the origin. Then $\sigma_\infty: \sigma_\infty^{-1}(M) \to M$ is an isomorphism.
	\end{lemma}
	\begin{proof} Let $E$ be the exceptional divisor.
	    Then
	    $\sigma_\infty: \LL(B)-\LL(E) \to \LL(\bC^2) - \{ 0 \}$
	    is a bijective morphism, so we only need to show $\sigma_\infty^{-1}: M \to \sigma_\infty^{-1}(M)$ is a morphism. We have $M = M_1 \cup M_2$, where	    $M_i = \{\gamma=(\gamma_1,\gamma_2) \in M \ | \ \ord_t \gamma_i = n \}$ is open in $M$ for $i=1,2.$ Using charts for $B$, the restriction of $\sigma_\infty^{-1}$ to $M_1$ and $M_2$ is given by $(\gamma_1,\gamma_2)\mapsto (\gamma_1,{\gamma_2}/{\gamma_1})$ and $({\gamma_1}/{\gamma_2},\gamma_2)$, respectively.
	    	    Computing these expressions explicitly, one sees they are morphisms, and hence, so is $\sigma_\infty^{-1}$.
	\end{proof}

	\begin{proposition}
		\label{last-branch}
		If $Z_g$ is nonempty, there exists precisely one divisor $E_{i_0}$ in the $g$-th branch such that $Z_g = \overline{\sX_{m,i_0}}$. This divisor is characterized as the one closest to the rupture component that gives an $m$-valuation (it could be the rupture component itself).
	\end{proposition}
	\begin{proof}
		By Lemma \ref{same-inters},	
		$Z_g=\left\{\gamma \in \X_m(C) \ \middle| \ \gamma^{(R_{g-1})} \cdot E_{R_{g-1}} = {m}/{(\kappa_g+N_{R_{g-1}})}\right\}.$ 
		So if $\gamma\in Z_g$, then
		$m=\gamma \cdot C=\gamma^{(R_{g-1})} \cdot C^{(R_{g-1})} + N_{R_{g-1}} (\gamma^{(R_{g-1})} \cdot E_{R_{g-1}})$.
		Hence
		$m':=\gamma^{(R_{g-1})} \cdot C^{(R_{g-1})}=m-N_{R_{g-1}}{m}/(\kappa_g+N_{R_{g-1}})$
		 is constant for all arcs in $Z_g$. 
		  Moreover, fixing $\gamma^{(R_{g-1})} \cdot E_{R_{g-1}}$ one fixes the sequence of multiplicities of the lifts of $\gamma$ up to $Y_{R_{g-1}-1}$. Applying Lemma \ref{equimultiplicity} repeatedly, we see that lifting arcs establishes a isomorphism between $Z_g$, the arcs of $\sX_m(C)$ that lift to the $g$-th branch, and the arcs of $\sX_{m'}(C^{(R_{g-1})})$  that lift to the corresponding part of the resolution graph of $C^{(R_{g-1})}$, that is, one of the sides of the first rupture component.	 
		 Since $C^{(R_{g-1})}$ has zero or one Puiseux pair, the result now follows from Proposition \ref{1pp}.
	\end{proof}
	
	\begin{prop}
		If $Z_j$ is nonempty, there exists precisely one divisor $E_{i_0}$ in the $j$-th branch such that  $Z_j = \overline{\sX_{m,i_0}}$. This divisor is characterized as the one closest to the rupture component that gives an $m$-valuation (it could be the rupture component itself).
	\end{prop}
	\begin{proof}
		By Lemma \ref{valuation-char-roots}, we have
		$\sX_{m,i}(C)=\sX_{{m}/{r_{j+1}},i}(C_j)$ for all $E_i$ no later than $E_{R_j}$.
		 In particular, since the $j$-th branch of the resolution graph of $C$ is the $j$-th (and therefore the last) branch of the resolution graph of $C_j$, the result follows from Proposition \ref{last-branch}.
	\end{proof}

\noi{\bf Proof of Theorem \ref{thmNPCur}.}  			(ii) This is true in general, by Proposition \ref{lemdlE}. Parts (i) and (iii) are proved in Theorem \ref{decomp} and the last three Propositions. 	
\hfill$\Box$

\end{document}

%% file: one-puiseux-pair.pdf_tex
\begingroup%
  \makeatletter%
  \providecommand\color[2][]{%
    \errmessage{(Inkscape) Color is used for the text in Inkscape, but the package 'color.sty' is not loaded}%
    \renewcommand\color[2][]{}%
  }%
  \providecommand\transparent[1]{%
    \errmessage{(Inkscape) Transparency is used (non-zero) for the text in Inkscape, but the package 'transparent.sty' is not loaded}%
    \renewcommand\transparent[1]{}%
  }%
  \providecommand\rotatebox[2]{#2}%
  \newcommand*\fsize{\dimexpr\f@size pt\relax}%
  \newcommand*\lineheight[1]{\fontsize{\fsize}{#1\fsize}\selectfont}%
  \ifx\svgwidth\undefined%
    \setlength{\unitlength}{656.78860402bp}%
    \ifx\svgscale\undefined%
      \relax%
    \else%
      \setlength{\unitlength}{\unitlength * \real{\svgscale}}%
    \fi%
  \else%
    \setlength{\unitlength}{\svgwidth}%
  \fi%
  \global\let\svgwidth\undefined%
  \global\let\svgscale\undefined%
  \makeatother%
  \begin{picture}(1,0.44010074)%
    \lineheight{1}%
    \setlength\tabcolsep{0pt}%
    \put(0,0){\includegraphics[width=\unitlength,page=1]{one-puiseux-pair.pdf}}%
    \put(-0.00176938,0.21342692){\color[rgb]{0,0,0}\makebox(0,0)[lt]{\lineheight{1.25}\smash{\begin{tabular}[t]{l}$E_{R_1}$\end{tabular}}}}%
    \put(0.2548622,0.247593){\color[rgb]{0,0,0}\makebox(0,0)[lt]{\lineheight{1.25}\smash{\begin{tabular}[t]{l}$E_{R_j}$\end{tabular}}}}%
    \put(0.66545772,0.247593){\color[rgb]{0,0,0}\makebox(0,0)[lt]{\lineheight{1.25}\smash{\begin{tabular}[t]{l}$E_{R_g}$\end{tabular}}}}%
    \put(0.45940837,0.24838231){\color[rgb]{0,0,0}\makebox(0,0)[lt]{\lineheight{1.25}\smash{\begin{tabular}[t]{l}$E_{R_{j+1}}$\end{tabular}}}}%
    \put(0,0){\includegraphics[width=\unitlength,page=2]{one-puiseux-pair.pdf}}%
    \put(0.89895792,0.21379758){\color[rgb]{0,0,0}\makebox(0,0)[lt]{\lineheight{1.25}\smash{\begin{tabular}[t]{l}$\tilde C$\end{tabular}}}}%
    \put(0.318194,0.17661614){\color[rgb]{0,0,0}\makebox(0,0)[lt]{\lineheight{1.25}\smash{\begin{tabular}[t]{l}\color{gray}$\mathscr{X}_{m,i}$\end{tabular}}}}%
    \put(0.00377712,0.07404936){\color[rgb]{0,0,0}\makebox(0,0)[lt]{\lineheight{1.25}\smash{\begin{tabular}[t]{l}\color{gray}$Z_1$\end{tabular}}}}%
    \put(0.20475513,0.07404936){\color[rgb]{0,0,0}\makebox(0,0)[lt]{\lineheight{1.25}\smash{\begin{tabular}[t]{l}\color{gray}$Z_j$\end{tabular}}}}%
    \put(0.39431401,0.07404936){\color[rgb]{0,0,0}\makebox(0,0)[lt]{\lineheight{1.25}\smash{\begin{tabular}[t]{l}\color{gray}$Z_{j+1}$\end{tabular}}}}%
    \put(0.6135616,0.07404936){\color[rgb]{0,0,0}\makebox(0,0)[lt]{\lineheight{1.25}\smash{\begin{tabular}[t]{l}\color{gray}$Z_{g}$\end{tabular}}}}%
  \end{picture}%
\endgroup%

%% file: milnor-fiber.pdf_tex
\begingroup%
  \makeatletter%
  \providecommand\color[2][]{%
    \errmessage{(Inkscape) Color is used for the text in Inkscape, but the package 'color.sty' is not loaded}%
    \renewcommand\color[2][]{}%
  }%
  \providecommand\transparent[1]{%
    \errmessage{(Inkscape) Transparency is used (non-zero) for the text in Inkscape, but the package 'transparent.sty' is not loaded}%
    \renewcommand\transparent[1]{}%
  }%
  \providecommand\rotatebox[2]{#2}%
  \newcommand*\fsize{\dimexpr\f@size pt\relax}%
  \newcommand*\lineheight[1]{\fontsize{\fsize}{#1\fsize}\selectfont}%
  \ifx\svgwidth\undefined%
    \setlength{\unitlength}{1533.43932625bp}%
    \ifx\svgscale\undefined%
      \relax%
    \else%
      \setlength{\unitlength}{\unitlength * \real{\svgscale}}%
    \fi%
  \else%
    \setlength{\unitlength}{\svgwidth}%
  \fi%
  \global\let\svgwidth\undefined%
  \global\let\svgscale\undefined%
  \makeatother%
  \begin{picture}(1,0.5713115)%
    \lineheight{1}%
    \setlength\tabcolsep{0pt}%
    \put(0,0){\includegraphics[width=\unitlength,page=1]{milnor-fiber.pdf}}%
  \end{picture}%
\endgroup%

%% file: mNash20240626.bbl
\begin{thebibliography}{KMM87}

\bibitem[AC75]{AC} N. A'Campo, {La fônction zeta d'une monodromie}, Comment. Math. Helv. 50, 233-248. (1975).

\bibitem[BPV84]{BPV} W. Barth, C. Peters, A. Van de Ven, {Compact complex surfaces.}  Springer-Verlag, Berlin, 1984. x+304 pp.

\bibitem[BV18]{BV} H. Baumers, W. Veys, {Contribution of jumping numbers by exceptional divisors.} With an appendix by Smith and Tucker. J. Algebra 496, 24-47 (2018).


\bibitem[BNS20]{BLM}  D. Bourqui, J. Nicaise, J. Sebag (eds.),
Arc schemes and singularities. World Scientific Publishing Co. Pte. Ltd., Hackensack, NJ, 2020. xi+299 pp.

\bibitem[BK86]{BK} E. Brieskorn, H. Kn\"orrer, {Plane algebraic curves.}  Birkh\"auser Verlag, Basel, 1986. vi+721 pp.

\bibitem[BMS13]{BMS} C. Bruschek, H. Mourtada, J. Schepers, {Arc spaces and the Rogers-Ramanujan identities.} Ramanujan J. 30,  9-38  (2013).

\bibitem[BFLN19]{Fl} N. Budur,  J. Fern\'andez de Bobadilla, Q.T. L\^e,  H. D. Nguyen, {Cohomology of contact loci.}  J. Differential Geom. 120, 389-409 (2022).

\bibitem[BT20]{BT} N. Budur, T.Q. Tue, {On contact loci of hyperplane arrangements.}  Adv. in Appl. Math. 132, 102271 (2022).

\bibitem[CNS18]{ACL} A. Chambert-Loir, J. Nicaise, J. Sebag, {Motivic integration.}
Birkh\"auser, New York, 2018. xx+526 pp.

\bibitem[CM21]{CoM} H. Cobo, H. Mourtada, {Jet schemes of quasi-ordinary surface singularities.} Nagoya Math. J. 242,  77-164 (2021).

\bibitem[dFD16]{dFD} T. de Fernex, R. Docampo, {Terminal valuations and the Nash problem.} Invent. Math. 203,  303-331 (2016).


\bibitem[DL98]{DL} J. Denef, F. Loeser, {Motivic Igusa zeta functions.} J. Algebraic Geom. 7,  505-537 (1998).

\bibitem[DL02]{DL02} J. Denef, F. Loeser, {Motivic Integration, Quotient Singularities and the McKay Correspondence.} Compositio Mathematica 131, 267–290 (2002). 

\bibitem[Do13]{Roi} R. Docampo, {Arcs on determinantal varieties.} Trans. Amer. Math. Soc. 365, 2241-2269 (2013).

\bibitem[ELM04]{ELM} L. Ein, R. Lazarsfeld, M. Musta\c{t}\u{a}, {
Contact loci in arc spaces.} Compos. Math. 140, 1229-1244  (2004). 

\bibitem[EM04]{EM} L. Ein, M. Musta\c{t}\u{a}, 
Inversion of adjunction for local complete intersection varieties. Amer. J. Math. 126, 1355-1365 (2004).

\bibitem[FP12]{FdBPP} J. Fern\'andez de Bobadilla, M. Pe Pereira, {The Nash problem for surfaces.} Ann. Math.  176, 2003-2029 (2012).

\bibitem[FPP17]{FPPP} J. Fern\'andez de Bobadilla, M. Pe Pereira, P. Popescu-Pampu, {On the generalised Nash problem of smooth germs and adjacencies of curve singularities}. Adv. Math. 320, 1269-1317 (2017).

\bibitem[Fu12]{Fu} O. Fujino, {Minimal model theory for log surfaces.} Publ. Res. Inst. Math. Sci. 48, 339-371 (2012).

\bibitem[HX13]{HX} C. Hacon,  C. Xu, {Existence of log canonical closures.} Invent. Math. 192, 161-195 (2013).

\bibitem[I04]{Is-t} S. Ishii, {The arc space of a toric variety.} J. Algebra 278, 666-683 (2004).

\bibitem[I08]{Is} S. Ishii, {Maximal divisorial sets in arc spaces.} Algebraic geometry in East Asia-Hanoi 2005, 237-249, Adv. Stud. Pure Math., 50, Math. Soc. Japan, Tokyo, 2008. 


\bibitem[K+20]{KMPT} B. Karadeniz, H. Mourtada, C. Pl\'enat, M. Tosun, {The embedded Nash problem of birational models of rational triple singularities.} J. Singul. 22,  337-372 (2020).

\bibitem[KMM87]{KMM}  Y. Kawamata, K. Matsuda,  K. Matsuki, {Introduction to the minimal
model problems}. Algebraic geometry, Sendai, 1985, 283-360, Adv. Stud. Pure Math., 10, North-Holland, Amsterdam, 1987.

\bibitem[K97]{Ko} J. Koll\'ar, {Singularities of pairs.} Algebraic geometry-Santa Cruz 1995, 221-287, Proc. Sympos. Pure Math., 62, Part 1, Amer. Math. Soc., Providence, RI, 1997.

\bibitem[K13]{K} J. Koll\'ar, {Singularities of the minimal model program.} With a collaboration of S\'andor Kov\'acs.  Cambridge Univ. Press, Cambridge, 2013. x+370 pp.

\bibitem[K+92]{K-ab} J. Koll\'ar et al., {Flips and abundance for algebraic threefolds.}  Ast\'erisque  211 (1992).


\bibitem[KM98]{KM} J. Koll\'ar, S. Mori, {Birational geometry of algebraic varieties}.  Cambridge Univ. Press, Cambridge, 1998. viii+254 pp.

\bibitem[Ko22]{Kor} Y. Koreeda, {On the configuration of the singular fibers of jet schemes of rational double points.} Comm. Algebra 50,  1802-1820 (2022).


\bibitem[LMR13]{LJMR} M. Lejeune-Jalabert, H. Mourtada, A. Reguera,  {Jet schemes and minimal embedded desingularization of plane branches.} Rev. R. Acad. Cienc. Exactas F\'is. Nat. Ser. A Mat. RACSAM 107, 145-157 (2013).

\bibitem[Mc19]{Mc} M. McLean, {Floer cohomology, multiplicity and the log canonical threshold.} Geom. Topol. 23, 957-1056 (2019).

\bibitem[Mo12]{Mou} H. Mourtada, {Jet schemes of complex plane branches and equisingularity.} Ann. Inst. Fourier (Grenoble) 61, 2313-2336 (2012). 

\bibitem[Mo14]{Mou-rat} H. Mourtada, {Jet schemes of rational double point singularities.} Valuation theory in interaction, 373--388, EMS Ser. Congr. Rep., Eur. Math. Soc., Z\"urich, 2014.

\bibitem[Mo17]{Mou-tor} H. Mourtada, {Jet schemes of normal toric surfaces.} Bull. Soc. Math. France 145, 237-266 (2017).

\bibitem[MoPl18]{Mou-Pl} H. Mourtada, C. Pl\'enat, {Jet schemes and minimal toric embedded resolutions of rational double point singularities.} Comm. Algebra 46,  1314-1332 (2018).

\bibitem[Mu01]{M1} M. Musta\c{t}\u{a}, {Jet schemes of locally complete intersection canonical singularities, with an appendix by D. Eisenbud and E. Frenkel.} Invent. Math. 145, 397-424 (2001).

\bibitem[Mu02]{M} M. Musta\c{t}\u{a}, {Singularities of pairs via jet schemes.} J. Amer. Math. Soc. 15,  599-615 (2002).

\bibitem[NS07a]{NS1} J. Nicaise, J. Sebag, {
Motivic Serre invariants of curves.} Manuscripta Math. 123, 105-132 (2007).

\bibitem[NS07b]{NS2} J. Nicaise, J. Sebag, {
Motivic Serre invariants, ramification, and the analytic Milnor fiber.}
Invent. Math. 168, 133-173 (2007).


\bibitem[OX12]{OX} Y. Odaka, C. Xu, {
Log-canonical models of singular pairs and its applications. }
Math. Res. Lett. 19, 325-334 (2012).

\bibitem[Pa24]{Pa} K. Palka, {Almost minimal models of log surfaces}. Preprint, 2024, arxiv:2402.07187. 

\bibitem[P03]{PP} P. Popescu-Pampu, {
Approximate roots.} Valuation theory and its applications, Vol. II (Saskatoon, SK, 1999), 285-321,
Fields Inst. Commun., 33, Amer. Math. Soc., Providence, RI, 2003.


\bibitem[Se96]{Se} P. Seidel, {The symplectic Floer homology of a Dehn twist.} Math. Res. Lett.  3, 829-834 (1996).
 
\bibitem[Te07]{Te} J. Tevelev, {Compactifications of subvarieties of tori.} Amer. J. Math. 129, 1087-1104 (2007).


\end{thebibliography}
